\documentclass[12pt, a4paper]{amsart}
\usepackage{amscd,amsmath,amssymb,amsthm,amsfonts}
\usepackage[alphabetic]{amsrefs}
\usepackage{mathtools,mathrsfs}
\usepackage[shortlabels]{enumitem}
\usepackage{pgf}
\usepackage[all]{xy}
\usepackage{tikz-cd}
\usetikzlibrary{positioning,shapes,arrows,chains,matrix}

\usepackage{xcolor}
\usepackage[hypertexnames=true, citecolor = blue, colorlinks = true, linkcolor = blue, linktoc = none, pdffitwindow = false, urlcolor=blue,urlbordercolor = white, pdftoolbar=false]{hyperref}

\def\Aut{\operatorname{Aut}}

\def\id{\operatorname{id}}

\def\id{\operatorname{id}}

\def\R{\mathbb{R}}
\def\N{\mathbb{N}}
\def\Z{\mathbb{Z}}

\newcommand{\CI}[0]{\mathcal{I}} 
 
 \newcommand{\CN}[0]{\mathcal{N}}

 \newcommand{\CT}[0]{\mathcal{T}}

\def\ssubset{\subset\hspace*{-1.7mm}\subset}

\def\gxp{G \rtimes_\theta P}

\newtheorem{thm}{Theorem}[section]
\newtheorem{corollary}[thm]{Corollary}
\newtheorem{lemma}[thm]{Lemma}

\newtheorem{proposition}[thm]{Proposition}

\theoremstyle{definition}
\newtheorem{definition}[thm]{Definition}
\newtheorem{notation}[thm]{Notation}

\theoremstyle{remark}
\newtheorem{remark}[thm]{Remark}

\newtheorem{example}[thm]{Example}

\numberwithin{equation}{section}

\begin{document}

\title{The nature of generalized scales}

\author[N.~Stammeier]{Nicolai Stammeier}
\address{Department of Mathematics \\ University of Oslo \\ P.O. Box 1053 \\ Blindern \\ NO-0316 Oslo \\ Norway}
\email{nicolsta@math.uio.no}

\thanks{This research was supported by RCN through FRIPRO 240362.}

\begin{abstract}
The notion of a generalized scale emerged in recent joint work with Afsar--Brownlowe--Larsen on equilibrium states on $C^*$-algebras of right LCM monoids, where it features as the key datum for the dynamics under investigation. This work provides the structure theory for such monoidal homomorphisms. We establish uniqueness of the generalized scale and characterize its existence in terms of a simplicial graph arising from a new notion of irreducibility inside right LCM monoids. In addition, the method yields an explicit construction of the generalized scale if existent. We discuss applications for graph products as well as algebraic dynamical systems, and reveal a striking connection to Saito's degree map. 
\end{abstract}

\date{\today}
\maketitle

\section{Introduction}
\noindent The notion of a generalized scale first appeared in \cite{ABLS} as an ad-hoc definition that allowed for an abstract and unified theory of equilibrium states for a number of quantum statistical mechanical systems \cites{LR2,BaHLR,LRR,LRRW,CaHR}. More precisely, all these case studies can be viewed as considerations for $C^*$-algebras of \emph{right LCM monoids}, that is, left cancellative monoids in which the intersection of two principal right ideals is either empty or another principal right ideal again. Here, right LCM refers to the existence of a \emph{right Least Common Multiple} given the presence of a right common multiple. We note that right LCM monoids are commonly referred to as right LCM semigroups in the operator algebraic literature as the presence of a unit is taken for granted in this context. Since this work is part of semigroup theory rather than operator algebras, we decided to correct this potentially misleading terminology.

In order to give the definition of a generalized scale, it is convenient to first recall some notation: Let $S$ be a right LCM monoid. For $s,t \in S$, we follow the notation of \cite{Spi1} and write $s\Cap t$ if $sS\cap tS$ is nonempty, and $s\perp t$ otherwise, in which case we call $s$ and $t$ orthogonal. An important submonoid of a right LCM monoid $S$ is its \emph{core} $S_c := \{a \in S \mid a\Cap s \text{ for all } s \in S\}$ introduced in \cite{Star}, inspired by \cite{CrispLaca}. As in \cite{ABLS}*{Subsection~2.1}, $s,t\in S$ are \emph{core equivalent}, denoted $s\sim t$, if there are $a,b \in S_c$ with $sa=tb$. It will be convenient to work with the following definition of a generalized scale that avoids the notion of accurate foundation sets:

\begin{definition}\label{def:gen scale}
A \emph{generalized scale} on a right LCM monoid $S$ is a nontrivial homomorphism $N\colon S \to \N^\times, s \mapsto N_s$ satisfying 
\begin{enumerate}[(i)]
\item $\lvert N^{-1}(n)/_\sim \rvert = n$ for all $n \in N(S)$;
\item $N_s=N_t$ implies $s\sim t$ or $s\perp t$; and
\item for all $s \in S, n\in N(S)$, there exists $t \in N^{-1}(n)$ with $s \Cap t$.
\end{enumerate}
\end{definition}

Proposition~\ref{prop:eq def of GS} will establish that this definition is equivalent to the original one given in \cite{ABLS}*{Definition~3.1(A3)}. The assumption that $N$ has to be nontrivial excludes precisely the left reversible right LCM monoids and is needed to obtain a nontrivial time evolution of the quantum statistical mechanical system based on $C^*(S)$, see \cite{ABLS}*{Section~4}.

Arguably, the definition of a generalized scale was useful, but rather elusive. It is thus natural to ask basic questions such as:
\begin{enumerate}[1.]
\item How many generalized scales may a right LCM monoid have?
\item When does a right LCM monoid admit a generalized scale?
\item Which structural aspects are captured by a generalized scale? 
\end{enumerate}
Already in \cite{Sta4}*{Corollary~4.9} it had been observed that right-angled Artin monoids have at most one generalized scale. These monoids were first considered in \cite{CF} (also known as \emph{free partially commutative monoids} or \emph{trace monoids}). This result can then be further exploited with the help of \cite{ABLS}*{Proposition~5.4} to show uniqueness of the generalized scale for self-similar group actions \cite{ABLS}*{Subsection~5.3} and Baumslag--Solitar monoids with positive parameters \cite{ABLS}*{Subsection~5.5}. Combining \cite{Sta4}*{Corollary~4.5} with \cite{ABLS}*{Proposition~3.6(v)}, it follows that there is an abundance of right LCM monoids without generalized scales.

In the present work, we start by answering the first question by establishing a uniqueness theorem for generalized scales, see Theorem~\ref{thm:uniqueness}. Its proof relies on the preservation of arithmetic aspects of $\N^\times$ under generalized scales. On the way to this result, we show that condition (A4) from \cite{ABLS}*{Definiton~3.1} is redundant, see Proposition~\ref{prop:A3 gives A4}. This is of particular relevance to \cite{BLRSt}, where, using this work, we obtain a classification of KMS-states for arbitrary right LCM monoids that admit a (necessarily unique) generalized scale. That is to say, the only requirement left for obtaining the conclusions of \cite{ABLS}*{Theorem~4.3} is the existence of a generalized scale. The guiding theme for this improvement is to work with the quotient space $S/_\sim$ and the action $\alpha\colon S_c \curvearrowright S/_\sim$ induced by left multiplication instead of using arithmetic structures inside the semigroup. In return, \cite{BLRSt} indicates that a very natural direction for further progress on KMS-states in the context of semigroup C*-algebras is a sophisticated understanding of generalized scales on right LCM monoids, which is the aim of the present work.

The answers to the second and third questions are given in Theorem~\ref{thm:char existence of the GS} and require more work: A central object in the proof of Theorem~\ref{thm:uniqueness} is the set of irreducible elements $\text{Irr}~N(S)$ of the submonoid $N(S) \subset \N^\times$. In Proposition~\ref{prop:indec = preimages of Irr(N(S))}, we note that their preimages under $N$ are precisely the elements $s \in S$ with the following property: whenever $s=tr$ for $t,r \in S$, then either $t \in S_c$ or $r \in S_c$. We shall call these elements \emph{noncore irreducible}, see Definition~\ref{def:non-core irred}, and denote their collection by $\CI(S)$. For a basic understanding of their behavior, we keep track of whether two noncore irreducible elements $s$ and $t$ satisfy $s \perp t$ or $s\Cap t$. Since this is a property of $[s],[t] \in S/_\sim$ rather than their representatives, we reduce to equivalence classes. By drawing an edge between $[s]$ and $[t]$ whenever $s\Cap t$, we obtain a simplicial graph $\Gamma(S)$ in Definition~\ref{def:core graph}, the \emph{core graph} of $S$. It turns out that the restriction of $\alpha$ to $\CI(S)/_\sim$ induces an action $\beta\colon S_c \curvearrowright \Gamma(S)$ by graph automorphisms, see Proposition~\ref{prop:alpha restricts to co-conn comp}, which may be of independent interest. The existence of a generalized scale implies that this action is simply a composition of permutations of the vertex sets of the coconnected components of $\Gamma(S)$, see Corollary~\ref{cor:alpha permutes co-conn comps} and Theorem~\ref{thm:char existence of the GS}~(i).

With the core graph at our disposal, we set out to characterize the existence of the generalized scale and to provide an explicit construction in Theorem~\ref{thm:char existence of the GS}. We find that, in addition to quite restrictive features of the core graph, see Theorem~\ref{thm:char existence of the GS}(i),(ii), the existence of a generalized scale also requires that 
\begin{enumerate}
\item[(iii)] $S$ is \emph{noncore factorable}, that is, every element in $S\setminus S_c$ be core equivalent to a finite product of noncore irreducibles; and that
\item[(iv)] $S$ has \emph{balanced factorization}: for $s,t \in \CI(S)$ with $s\Cap t$, there are $s',t' \in \CI(S)$ such that $sS\cap tS = st'S, st'=ts'$ with $[s']$ and $[t']$ belonging to the same coconnected component of $\Gamma(S)$ as $[s]$ and $[t]$, respectively. 
\end{enumerate}
As to be expected, showing that the conditions (i)--(iv) are not only necessary, but in fact sufficient for the existence of a generalized scale is the hard part of the proof. One of the key tools in the proof is an algorithm for computing a right LCM of products of noncore irreducibles, see Remark~\ref{rem:rLCM for products built from I(S)}.

In the final section, we first describe the core graph and its features for the two elementary examples of self-similar group actions in \ref{subsec:ex1} and the ax$+$b-semigroup over the natural numbers in \ref{subsec:ex2}. For algebraic dynamical systems, that is, suitable actions of right LCM monoids on discrete groups by injective group endomorphisms, the application of Theorem~\ref{thm:char existence of the GS} is less straightforward. We examine the sufficient criterion for the existence of a generalized scale from \cite{ABLS}*{Proposition~5.11(i)}, and find that it is not necessary in general, see Example~\ref{ex:ads with free monoid}. On the other hand, Ledrappier's shift is an example that appears to lack a generalized scale precisely for the violation of condition (b) of Remark~\ref{rem:I(GxP)=GxI(P)}, see Example~\ref{ex:Ledrappier shift}. Algebraic dynamical systems also provide the flexibility to build examples where $\beta$ switches coconnected components of $\Gamma(S)$, see Example~\ref{ex:Z2 xp with flip}. 

Motivated by the findings in \cite{Sta4}, we investigate applications to graph products of right LCM monoids, and obtain a result which reduces the considerations to the coconnected case, see Theorem~\ref{thm:gen scales for graph products}. However, already Example~\ref{ex:graph products gone mad} shows that graph products of coconnected graphs may well behave counterintuitively, depending on the combination of edge set and family of vertex semigroups. Roughly speaking, very few graph products admit generalized scales. This even remains true for the most elementary examples of this type, namely right-angled Artin monoids, see \cite{Sta4}*{Corollary~4.9}. 

In summary, the present work shows that the generalized scale represents an intriguing set of characteristic features of the right LCM monoid in question, see Theorem~\ref{thm:char existence of the GS}. These features are common for many types of right LCM monoids, see \cite{ABLS}. However, there are also many interesting examples where different parts of the requirements for the existence of the generalized scale are not met. This motivates the search for a more flexible notion of such maps, the outcome of which we present in Subsection~\ref{subsec:saito}: Generalized scales correspond to special kinds of Saito's degree maps, see \cite{Sai} for details.\vspace*{1mm}

\noindent\emph{Acknowledgements:} The author is grateful to Nathan Brownlowe, Nadia S.~Larsen, Jacqui Ramagge, and Jack Spielberg for valuable discussions. He wishes to thank the anonymous referee for valuable comments leading to an improved exposition of the material, in particular for suggesting an alternative proof to Proposition~\ref{prop:A3 gives A4}, see Remark~\ref{rem:alt proof by referee}, and for spotting details in the proofs of Lemma~\ref{lem:permute in prod of irreds-new}~(i) and Theorem~\ref{thm:char existence of the GS} that needed extra care.

\section{Preliminaries}\label{sec:prelim}
In an attempt to make this work essentially self-contained, we shall briefly recall the auxiliary results from \cite{ABLS} that we will need and make a few additional observations to help the reader familiarize with the concepts of this work. Throughout, let $S$ be a right LCM monoid. Let us start with the first part of \cite{ABLS}*{Lemma~3.5}:

\begin{lemma}\label{lem:cap with core}
Let $a \in S_c$ and $s \in S$. If $b \in S$ satisfies $aS\cap sS = sbS$, then $b \in S_c$.
\end{lemma}
\begin{proof}
Let $a \in S_c$ and $s,b \in S$ satisfy $sS\cap aS = sbS$ and fix $t\in S$. Then 
\[s(tS \cap bS)= stS\cap sS \cap aS = stS \cap aS \neq \emptyset\]
shows $t\Cap b$, and hence $b\in S_c$ since $t$ was arbitrary.
\end{proof}

Recall that a subsemigroup $T\subset S$ is said to be \emph{hereditary} if it has the following property: If $s \in S$ satisfies $t \in sS$ ($t\geq s$) for some $t \in T$, then $s \in T$. 

\begin{lemma}\label{lem:rLCM if core eq}
The core $S_c$ is a hereditary submonoid of $S$. In particular, if $s,t \in S$ satisfy $s\sim t$, then $sS\cap tS=saS$ for some $a \in S_c$.
\end{lemma}
\begin{proof}
For the first claim, we note that for $a \in S, b \in S_c$ with $b \in aS$, we have $aS\cap sS \supset bS \cap sS \neq \emptyset$ for all $s \in S$, so that $a \in S_c$. For the second part, $s\sim t$ implies that there are $b,c \in S_c$ with $sb=tc$. So we have $sS\cap tS=saS$ for some $a \in S$ with $b \in aS$. Thus the first part yields $a\in S_c$.
\end{proof}

\begin{lemma}\label{lem:char of core eq}
Let $s,t \in S$. Then $s\sim t$ holds if and only if $s\Cap r \Leftrightarrow t\Cap r$ for all $r \in S$.
\end{lemma}
\begin{proof}
Let $s,t \in S$. Suppose first that $s\sim t$. Lemma~\ref{lem:rLCM if core eq} shows that $sS\cap tS = saS$ with $sa=tb$ for some $a,b \in S_c$. For $r \in S$, we thus get $r\Cap s \Leftrightarrow r\Cap sa=tb \Leftrightarrow r\Cap t$.

Conversely, assume that $s\Cap r \Leftrightarrow t\Cap r$ for all $r \in S$. In particular, picking $r:=s$ yields $s\Cap t$, say $sS\cap tS=saS$ with $sa=tb$ for some $a,b \in S$. We claim that $a$ and $b$ belong to $S_c $. For $p \in S$, let us consider $sp$. As $sp \Cap s$, we have $sp \Cap t$, which allows us to conclude
\[s(pS\cap aS) = spS\cap sS\cap tS = spS \cap tS \neq \emptyset.\]
Hence we obtain $a\Cap p$ for all $p \in S$, which is $a \in S_c$. Arguing in the same way for $b$, we also get $b \in S_c$.
\end{proof}

Following the terminology of \cite{SY}, the notion of a \emph{foundation set} was introduced in \cite{BRRW} for right LCM monoids: A foundation set is a finite subset $F\subset S$ (we write $F\ssubset S$ to indicate this) such that for every $s \in S$ there is $t \in F$ with $s\Cap t$. This was refined in \cite{BS1} to the notion of an \emph{accurate foundation set}, which is a foundation set $F$ satisfying $s\perp t$ for all $s,t \in F, s\neq t$. 
 
According to \cite{ABLS}*{Definition~3.1(A3)}, a nontrivial homomorphism $N\colon S \to \N^\times$ is called a \emph{generalized scale} if 
\begin{enumerate}[(a)]
\item $\lvert N^{-1}(n)/_\sim \rvert = n$ for all $n \in N(S)$; and
\item for each $n\in N(S)$, every transversal for $N^{-1}(n)/_\sim$ is an accurate foundation set for $S$.
\end{enumerate} 
Let us recall some key properties for generalized scales, namely \cite{ABLS}*{Proposition~3.6(i)-(iv)} without proof.

\begin{proposition}\label{prop:ABLS3.6} 
Let $N\colon S \to \N^\times$ be a generalized scale in the sense of \cite{ABLS}*{Definition~3.6(A3)}.
\begin{enumerate}[(i)]
\item $N^{-1}(1) = S_c$, and this is a proper subsemigroup of $S$.
\item For $s,t \in S$ with $N_s=N_t$, either $s\sim t$ or $s\perp t$ holds. 
\item Let $F$ be a foundation set for $S$ with $F \subset N^{-1}(n)$ for some $n \in N(S)$. If $\lvert F \rvert = n$ or $F$ is accurate, then $F$ is a transversal for $N^{-1}(n)/_\sim$.
\item Whenever $s,t \in S$ satisfy $sS\cap tS=rS$, then $N_sN(S)\cap N_tN(S) = N_rN(S)$. 
\end{enumerate}
\end{proposition}

In particular, part (i) implies that core equivalent elements have the same image under $N$. Let us now show that the two competing notions for a generalized scale coincide: 
 
\begin{proposition}\label{prop:eq def of GS} 
 A nontrivial homomorphism $N\colon S \to \N^\times$ is a generalized scale in the sense of \cite{ABLS}*{Definition~3.6(A3)} if and only if it satisfies (i)--(iii) of Definition~\ref{def:gen scale}.
 \end{proposition}
 \begin{proof}
 As (a) and (i) are identical, we need to show that (b) is equivalent to (ii) and (iii), given (i). We observe that (b) implies (iii) by the definition of a foundation set. Proposition~\ref{prop:ABLS3.6}(ii) establishes (ii). 
 
Conversely, suppose (i)--(iii) of Definition~\ref{def:gen scale} hold. We need to prove (b), so let $n \in N(S)$ and fix a transversal $F$ for $N^{-1}(n)/_\sim$. The set $F$ is finite by (i) and its elements are mutually orthogonal by (ii). Thus the claim reduces to showing that $F$ is a foundation set. For every $s \in S$, (iii) states that there is $t \in N^{-1}(n)$ with $s\Cap t$. If $t' \in F$ is the (unique) element satisfying $t'\sim t$, then Lemma~\ref{lem:char of core eq} shows $s\Cap t'$. Thus $F$ is an accurate foundation set, that is, (b) holds.
\end{proof}

\begin{corollary}\label{cor:N(S) is LCM}
For every generalized scale $N$, its image $N(S)$ is a (right) LCM subsemigroup of $\N^\times$.
\end{corollary}
\begin{proof}
Let $m,n \in N(S)$. So there is $s \in N^{-1}(m)$, and Definition~\ref{def:gen scale}~(iii) asserts that there is $t \in N^{-1}(n)$ with $s\Cap t$, that is, $sS\cap tS = rS$ for some $r\in S$. Using Proposition~\ref{prop:ABLS3.6}~(iv), we thus get $mN(S) \cap nN(S) = N_rN(S)$, so that $N_r$ is the (unique right) LCM of $m$ and $n$.
\end{proof}

Finally, let us recall the part of \cite{ABLS}*{Lemma~3.9} that holds without assumptions on \emph{core irreducible} elements in $S$, along with its proof.

\begin{lemma}\label{lem:thebijectionsg}
Let $S$ be a right LCM monoid. Then left multiplication defines an action $\alpha\colon S_c \curvearrowright S/_\sim$ by bijections $\alpha_a([s]) = [as]$. Every generalized scale $N$ on $S$ is invariant under this action.
\end{lemma}
\begin{proof}
For every $a \in S_c$, the map $\alpha_a$ is well-defined as left multiplication preserves the core equivalence relation. If $as \sim at$ for some $s,t \in S$, then $s \sim t$ by left cancellation. Hence $\alpha_a$ is injective. On the other hand, Lemma~\ref{lem:cap with core} states that for $s \in S$ we have $aS\cap sS = atS, at=sb$ for some $b \in S_c$ and $t \in S$. Thus $\alpha_a([t]) = [at] = [s]$, and we conclude that $\alpha_a$ is a bijection. Generalized scales are invariant under this action by Proposition~\ref{prop:ABLS3.6}~(i).
\end{proof}

\section{Uniqueness of the generalized scale}\label{sec:EaU}
Let us start with a few observations that will streamline the proof of the uniqueness theorem. For an LCM subsemigroup $\CN$ of $\N^\times$, let us denote its irreducible elements by 
\[ \text{Irr}~\CN := \{ n \in \CN\setminus \{1\} \mid n=k\ell \text{ for } k,\ell \in \CN \Rightarrow k=1 \text{ or } \ell=1\}. \] 

\begin{proposition}\label{prop:A3 gives A4}
Every LCM subsemigroup $\CN$ of $\N^\times$ is freely generated by its irreducible elements. In particular, $N(S)$ is freely generated by $\text{Irr}~N(S)$ for every generalized scale $N\colon S \to \N^\times$, that is, condition (A3) of \cite{ABLS}*{Definition~3.1} implies (A4). 
\end{proposition}
\begin{proof}
First, we claim that $m\CN \cap n\CN = mn\CN$ whenever $m,n \in \text{Irr}~\CN$ are distinct. There are unique $k,\ell \in \CN$ with $m\CN\cap n\CN = mk\CN$ and $n = k\ell$ because $mn \in mk\CN$. As $m\neq n$ forces $k>1$ and $\CN$ is cancellative, irreducibility of $n$ forces $\ell=1$, that is, $k=n$. 

For the general case suppose to the contrary that there were $k,\ell \geq 1$ and 
\[\quad m_1,\ldots,m_k,n_1,\ldots,n_\ell \in \text{Irr}~\CN \text{ with } m_1\cdots m_k = n_1\cdots n_\ell,\] 
but there did not exist a bijection $f\colon \{1,\ldots,k\} \to \{1,\ldots,\ell\}$ such that $m_{f(j)} = n_j$ for all $1\leq j\leq k$. By counting multiplicities of the elements in $\{m_j \mid 1\leq j\leq k\} \cap \{n_j \mid 1\leq j\leq \ell\}$, we can choose subsets $K \subset \{1,\ldots,k\}$ and $L \subset \{1,\ldots,\ell\}$ and build a bijection $f'\colon K \to L$ with $m_{f(j)} = n_j$ for all $j \in K$. By choosing $K$ maximal with the above property and removing it from $\{1,\ldots,k\}$ (and the corresponding $L$ from $\{1,\ldots,\ell\}$), the problem reduces to the following situation: There exist $k',\ell' \geq 1$ ($k=\ell$ and $K=\{1,\ldots,k\}$ cannot occur) and 
\[m_1,\ldots,m_{k'},n_1,\ldots,n_{\ell'} \in \text{Irr}~\CN \text{ with }
m_1\cdots m_{k'} = n_1\cdots n_{\ell'} \text{ and } m_i \neq n_j \text{ for all } i,j.\] 
In view of the first part of the proof, $n_1 \neq m_j$ for all $j$ gives
\[\begin{array}{lclcl}
n_1\CN \cap m_1\cdots m_{k'}\CN &=& n_1\CN \cap m_1\CN \cap m_1\cdots m_{k'}\CN \\
&=& m_1n_1\CN \cap m_1m_2\CN \cap m_1\cdots m_{k'}\CN \\
&=& m_1(n_1\CN\cap m_2\CN) \cap m_1\cdots m_{k'}\CN \\
&\vdots&\\
&=& m_1\cdots m_{k'}n_1\CN.
\end{array}\] 
As $n_1>1$, this contradicts 
\[n_1\CN \cap m_1\cdots m_{k'}\CN \supset n_1\cdots n_{\ell'}\CN \cap m_1\cdots m_{k'}\CN = m_1\cdots m_{k'}\CN.\] 
Hence we conclude that $\text{Irr}~\CN$ freely generates a submonoid $M$ of $\CN$. To see that $M=\CN$, we show by induction on $n\in \CN$ that every element in $\N$ factors into a finite product of elements from $\text{Irr}~\CN$: Suppose $\CN\neq \{1\}$ and pick the smallest $n \in \CN\setminus\{1\}$ with respect to the total order $\leq_\N$ coming from the canonical inclusion $\CN\subset \N$. Then $n \in \text{Irr}~\CN$ as $n=mm'$ forces $m,m' \in \{1,n\}$. Next, assume that a factorization with respect to $\text{Irr}~\CN$ exists for all $m\in \CN, m\leq_\N n$ for an arbitrary but fixed $n \in \CN$. Let $n' \in \CN$ be the smallest number with $n'>_\N n$. If $n' \in \text{Irr}~\CN$, then there is nothing to prove. On the other hand, if $n'=mm'$ for some $m,m' \in \CN\setminus \{1\}$, then $m,m' \leq_\N n$ since $n'$ was chosen to be minimal with $n'>_\N n$. Therefore, both $m$ and $m'$ admit a factorization with respect to $\text{Irr}~\CN$, and the product of these is a factorization for $n'$. This completes the induction and we conclude that $\CN$ is freely generated by $\text{Irr}~\CN$. 

The claim concerning a generalized scale $N$ now follows from Corollary~\ref{cor:N(S) is LCM}.
\end{proof}

\begin{remark}\label{rem:alt proof by referee}
The following elegant alternative for the proof of Proposition~\ref{prop:A3 gives A4} was suggested by the anonymous referee: It is known that an abelian cancellative monoid is free if and only if it
\begin{enumerate}[(a)]
\item has trivial group of units;
\item satisfies the ascending chain condition on principal ideals; and
\item each irreducible element is prime.
\end{enumerate}
In the present case of an LCM subsemigroup $\CN\subset \N^\times$, (a) and (b) are inherited from $\N^\times$. These two conditions imply that every element admits a factorization into a finite product of irreducibles. 

In order to establish (c), it suffices to show that if $p \in \text{Irr}~\CN$ and $q_1^{m_1}\cdots q_k^{m_k} \in p\CN$ for some mutually distinct $q_i \in \text{Irr}~\CN$ and $m_i\geq 1$ for all $i$, then $q_i \in p\CN$ for some $i$, that is, $p=q_i$ as (a) holds and $q_i$ is irreducible. Suppose that $p \neq q_i$ for all $i$ and that $k$ and $\sum_{1\leq i \leq k} m_i$ are minimal, that is, $q_1^{m'_1}\cdots q_\ell^{m'_\ell} \in p\CN$ implies $\ell\geq k$ and $\sum_{1\leq i \leq \ell} m'_i \geq \sum_{1\leq i \leq k} m_i$. Combining the first claim in the proof of Proposition~\ref{prop:A3 gives A4} with left cancellation, we obtain
\[\begin{array}{lcl}
p\CN \cap q_1\CN \cap \ldots \cap q_k\CN &=& pq_1\CN\cap pq_2\CN \cap \ldots \cap pq_k\CN \\
&=& p(q_1\CN\cap q_\CN\cap \ldots\cap q_k\CN)\\
&\vdots&\\
&=& pq_1q_2\cdots q_k\CN.
\end{array}\]
Thus we get 
\[q_1^{m_1}\cdots q_k^{m_k} \in p\CN \cap q_1\CN \cap \ldots \cap q_k\CN = pq_1q_2\cdots q_k\CN,\]
which implies $q_1^{m_1-1}\cdots q_k^{m_k-1} \in p\CN$ by left cancellation. This contradicts minimality of $(m_1,\ldots,m_k)$. 
\end{remark}

\begin{lemma}\label{lem:irr maps to irr}
Let $S$ be a right LCM monoid, $M,N\colon S \to \N^\times$ be two generalized scales on $S$, and $s \in S\setminus S_c$. If $M_s$ is reducible in $M(S)$, then there are $a\in S_c$ and $t,r \in S\setminus S_c$ such that $sa=tr$. In particular, $N_s$ is irreducible if and only if $M_s$ is irreducible. 
\end{lemma}
\begin{proof}
Suppose there are $k,\ell \in M(S), k,\ell > 1$ such that $M_s=k\ell$. By Definition~\ref{def:gen scale}~(iii), there is $t \in M^{-1}(k)$ such that $t\Cap s$, say $tS\cap sS = trS$ with $tr=sa$ for some $r,a \in S$. Due to Proposition~\ref{prop:ABLS3.6}~(iv) and $M_s=M_t\ell$, we know that $M_r=\ell>1$ and $M_a=1$. This implies $r \in S\setminus S_c$ and $a \in S_c$, see Proposition~\ref{prop:ABLS3.6}~(i). We conclude from this that $N_s=N_{sa} = N_tN_r$ with $N_t,N_r >1$, so that $N_s$ is reducible in $N(S)$.
\end{proof}

Recall from \cite{BS1}*{Definition~1.2 and Definition~2.1} that a finite subset $\CT\subset S$ is an \emph{accurate foundation set} if its elements are mutually orthogonal and for every $s \in S$ there is $t \in \CT$ with $s\Cap t$.

\begin{lemma}\label{lem:s Cap t for i(s) neq i(t)}
Let $N$ be a generalized scale on $S$. If $s,t \in S$ with $N_s,N_t \in \text{Irr}~N(S)$ satisfy $N_s\neq N_t$, then $s\Cap t$. 
\end{lemma}
\begin{proof}
According to Proposition~\ref{prop:A3 gives A4}, the right LCM of $N_s$ and $N_t$ in $N(S)$ is $N_sN_t$. So if $s' \in N^{-1}(N_s)$ and $t' \in N^{-1}(N_t)$ satisfy $s'S\cap t'S=rS$ for some $r\in S$, then $N_r=N_sN_t$ by Proposition~\ref{prop:ABLS3.6}~(iv).

Next, take transversals $F_r$ for $N^{-1}(N_r)/_\sim$ with $r \in F_r$ for $r=s,t$. By \cite{ABLS}*{Definition~3.1(A3)(b)}, $F_s$ and $F_t$ are accurate foundation sets with $\lvert F_r\rvert =N_r$ for $r=s,t$. Hence, every complete minimal set of representatives $F$ for 
\[ \{s'S\cap t'S \mid s'\in F_s, t'\in F_t\}\setminus \{\emptyset\}\]
is also an accurate foundation set. By the first paragraph, we have $F \subset N^{-1}(N_sN_t)$. According to Proposition~\ref{prop:ABLS3.6}~(iii), $F$ is a transversal for $N^{-1}(N_sN_t)/_\sim$, and thus $\lvert F\rvert =N_sN_t$. In view of $\lvert F_s\rvert = N_s$ and $\lvert F_t\rvert = N_t$, this forces $s\Cap t$.
\end{proof}

\begin{lemma}\label{lem:equiv prod of nc irreds}
Let $N$ be a generalized scale on the right LCM monoid $S$ and $s \in S\setminus S_c$. Then there are $k\geq 1$ and $s_1,\ldots,s_k \in S$ with $N_{s_j} \in \text{Irr}~N(S)$ for all $j$ satisfying $s\sim s_1\cdots s_k$.
\end{lemma}
\begin{proof}
Let $N_s=n_1\cdots n_k$ with $n_i \in \text{Irr}~N(S)$ for all $i$ be the (up to permutation unique) factorization into irreducibles in $N(S)\subset \N^\times$. By Definition~\ref{def:gen scale}(iii), there is  $s_1 \in N^{-1}(n_1)$ with $s_1\Cap s$, say $s_1S\cap sS=s_1r_1S, s_1r_1=sa_1$ for some $a_1,r_1\in S$. Since $N_{s_1}$ is a divisor of $N_s$, Proposition~\ref{prop:ABLS3.6}~(iv) implies that $N_{r_1}=n_2\cdots n_k$ and $N_{a_1}=1$. The latter is equivalent to $a_1\in S_c$, see Proposition~\ref{prop:ABLS3.6}~(i). Thus we obtained $s_1 \in N^{-1}(n_1),a_1 \in S_c$, and $r_1 \in N^{-1}(n_2\cdots n_k)$ such that 
\[sS\cap s_1S = s_1r_1S, s_1r_1=sa_1.\] 
Applying the same process to $(r_1,n_2)$ in place of $(s,n_1)$ yields $s_2\in  N^{-1}(n_2),a_2 \in S_c$, and $r_2 \in N^{-1}(n_3\cdots n_k)$ such that $r_1S\cap s_2S = s_2r_2S, s_2r_2=r_1a_2$. After $k-1$ steps, we have obtained $s_i \in N^{-1}(n_i), a_i \in S_c$ and $r_i \in N^{-1}(\prod_{i+1\leq j\leq k}n_j)$ for $1\leq i \leq k-1$ with
\[sa_1\cdots a_{k-1} = s_1\cdots s_{k-1}r_{k-1} \quad \text{with } r_{k-1} \in N^{-1}(n_k).\]
Thus, setting $s_k:= r_{k-1}$ gives $s\sim s_1\cdots s_k$.
\end{proof}

\begin{thm}\label{thm:uniqueness}
A right LCM monoid admits at most one generalized scale. 
\end{thm}
\begin{proof}
Let $S$ be a right LCM monoid and $M,N\colon S \to \N^\times$ be two generalized scales. Fix $s \in S$ with $N_s \in \text{Irr}~N(S)$. By Lemma~\ref{lem:irr maps to irr}, we know that $M_s \in \text{Irr}~M(S)$. We claim that $N_s=M_s$. According to Definition~\ref{def:gen scale}(i), this amounts to $\lvert N^{-1}(N_s)/_\sim\rvert = \lvert M^{-1}(M_s)/_\sim\rvert$. By symmetry, it thus suffices to show $N^{-1}(N_s)/_\sim \subset M^{-1}(M_s)/_\sim$. Suppose to the contrary that there exists $t \in N^{-1}(N_s)$ with $t\not\sim s$ and $M_t\neq M_s$. By Definition~\ref{def:gen scale}(ii), $N_t=N_s$ and $t\not\sim s$ force $s\perp t$. However, Lemma~\ref{lem:irr maps to irr} entails $M_s,M_t \in \text{Irr}~M(S)$ as $N_s=N_t \in  \text{Irr}~N(S)$. Therefore, Lemma~\ref{lem:s Cap t for i(s) neq i(t)} applied to $M$ gives $s\Cap t$, contradicting $s\perp t$. Thus we get $N^{-1}(N_s)/_\sim \subset M^{-1}(M_s)/_\sim$, and hence $N_s=M_s$ for every $s \in S$ with $N_s \in \text{Irr}~N(S)$. 

In addition, we know that $M$ and $N$ are homomorphisms with $M|_{S_c} = 1 = N|_{S_c}$. To conclude that $N=M$, let $s \in S\setminus S_c$. By Lemma~\ref{lem:equiv prod of nc irreds}, there are $k\geq 1$ and $s_1,\ldots,s_k \in S$ with $N_{s_i} \in \text{Irr}~N(S)$ such that $s\sim s_1\cdots s_k$. This allows us to deduce 
\[N_s= N_{s_1}\cdots N_{s_k} = M_{s_1}\cdots M_{s_k} = M_s,\]
that is, $M=N$.
\end{proof}

\section{Existence and construction of generalized scales}
The idea behind the final step in the proof of the uniqueness result Theorem~\ref{thm:uniqueness} will serve as our starting point: If a right LCM monoid $S$ admits a generalized scale $N\colon S \to \N^\times$, then every element in $S\setminus S_c$ is a nonempty, finite product of elements in $N^{-1}(\text{Irr}~N(S))$. We can characterize these minimal elements in the following way:

\begin{proposition}\label{prop:indec = preimages of Irr(N(S))}
Let $N\colon S \to \N^\times$ be a generalized scale. Then $s \in S\setminus S_c$ satisfies $N_s \in \text{Irr}~N(S)$ if and only if $sa=tr$ for $a \in S_c$ and $t,r \in S$ implies either $t \in S_c$ or $r \in S_c$.
\end{proposition}
\begin{proof}
First we note that $N_s \in \text{Irr}~N(S)$ requires $N_s>1$, which amounts to $s \in S\setminus S_c$, see Proposition~\ref{prop:ABLS3.6}(i). Likewise, $sa=tr$ with $a\in S_c$ and $t \notin S_c$ or $r \notin S_c$ forces $s \notin S_c$. Thus we can restrict our attention to $s \in S\setminus S_c$. In case there are $a \in S_c$ and $t,r \in S\setminus S_c$ such that $sa=tr$, Proposition~\ref{prop:ABLS3.6}(i) entails that $N_s = N_{sa} = N_tN_r$ with $N_t,N_r >1$, which means $N_s \notin \text{Irr}~N(S)$. Conversely, if $N_s \notin \text{Irr}~N(S)$, then there are $t,r \in S$ with $t,r \in S\setminus S_c$ and $a \in S_c$ such that $sa=tr$, see Lemma~\ref{lem:irr maps to irr}.
\end{proof}

\begin{definition}\label{def:non-core irred}
Let $S$ be a right LCM monoid. An element $s \in S$ is \emph{noncore irreducible} if $s \notin S_c$ and whenever $sa=tr$ for some $a \in S_c$ and $t,r \in S$, then $t \in S_c$ or $r\in S_c$. The collection of all noncore irreducible elements of $S$ is denoted by $\CI(S)$.
\end{definition}

\begin{remark}\label{rem:non-core irred vs not core irred}
The notion introduced in Definition~\ref{def:non-core irred} is not to be confused with \emph{core irreducibility} from \cite{ABLS}: $s \in S$ is core irreducible if $s=ta$ with $t \in S, a \in S_c$ implies $a \in S^*$, the subgroup of invertible elements in $S$. We note that noncore irreducibility is preserved under core equivalence. 
\end{remark}

\begin{definition}\label{def:core graph}
The \emph{core graph} of a right LCM monoid $S$ is the undirected graph $\Gamma(S) :=(V,E)$ with vertex set $V:= \CI(S)/_\sim$ and edge set $E:= \{ ([s],[t]) \mid s \Cap t\}$. 
\end{definition}

We will simply write $\Gamma$ whenever this is unambiguous. For convenience, we recall that $\Gamma$ is characterized by its unique decomposition into coconnected components $(\Gamma_i)_{i \in I}$ with $\Gamma_i = (V_i,E_i)$. Such coconnected components are determined by the vertex set $V_i$ as $E_i = V_i\times V_i \cap E$. A natural characterization of coconnectedness is that the $V_i$ forms a maximal, connected subset of the complimentary graph $(V,V\times V\setminus E)$.

\begin{notation}\label{not:co-conn comp identifier}
For $s \in \CI(S)$, we denote by $i(s) \in I$ the index with $[s] \in V_{i(s)}$.
\end{notation}

Recall from Lemma~\ref{lem:thebijectionsg} that $a.[s] := [as]$ defines an action $\alpha\colon S_c\curvearrowright S/_\sim$ by bijections.

\begin{lemma}\label{lem:alpha on I(S)}
For $a \in S_c$ and $s\in S$, $as \in \CI(S)$ holds if and only if $s \in \CI(S)$, that is, the action $\alpha$ restricts to an action on $\CI(S)/_\sim$.
\end{lemma}
\begin{proof}
Let $a \in S_c$ and $s \in S$. It is clear that $as \in S_c$ if and only if $s \in S_c$, so we may restrict to the case where $s \in S\setminus S_c$. We will prove that $\alpha_a^{-1}([s]) \in \CI(S)/_\sim$ if and only if $[s] \in \CI(S)/_\sim$, which is equivalent to the original claim since $\CI(S)$ is closed under core equivalence (essentially by definition) and $\alpha_a$ is bijective, see Lemma~\ref{lem:thebijectionsg}. Using Lemma~\ref{lem:cap with core}, there are $b \in S_c$ and $s' \in S\setminus S_c$ such that $sS\cap aS = sbS, sb=as'$. We claim that $s \in \CI(S)$ if and only if $s' \in \CI(S)$, that is, $[s] \in \CI(S)/_\sim$ if and only if $[s'] = \alpha_a^{-1}([s]) \in \CI(S)$: If there are $c \in S_c, t,r \in S\setminus S_c$ with $s'c=tr$, then $sbc=as'c=(at)r$ implies $s \notin \CI(S)$. Thus $s \in \CI(S)$ implies $s' \in \CI(S)$. Conversely, if there are $c \in S_c, t,r \in S$ with $sc=tr$, then $bS_c\cap cS_c = bc'S_c, bc'=cb'$ for some $b',c' \in S_c$ by Lemma~\ref{lem:cap with core}. With $trb'=scb'=as'c'$ we get a diagram
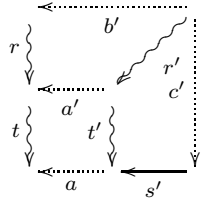
\[\xymatrix{
\ar@{~>}[d]_r && \ar@{.>}[ll]^{b'} \ar@{~>}[dl]^{r'} \ar@{.>}[dd]_{c'}\\
\ar@{~>}[d]_t & \ar@{.>}[l]^{a'} \ar@{~>}[d]_{t'}\\
&\ar@{.>}[l]^a & \ar^{s'}[l]
}\]
where $tS\cap aS=ta'S, ta'=at'$ and $r'\in S$ with $rb'=a'r'$. Here $r'$ results from the fact that $ta'$ is the right LCM of $t$ and $a$, while $trb'=as'c'$ is another common right multiple. The dotted arrows represent elements from $S_c$, while solid arrows refer to $\CI(S)$ (and $\rightsquigarrow$ bear no constraints). We get $s'c'=t'r'$ (upon using left cancellation). Since $a \in S_c$, we have $a' \in S_c$ by Lemma~\ref{lem:cap with core}. Thus $t \in S_c$ holds iff $t' \in S_c$, and $r \in S_c$ iff $r' \in S_c$. Therefore $s' \in \CI(S)$ forces $s \in \CI(S)$.
\end{proof}

\begin{proposition}\label{prop:alpha restricts to co-conn comp}
The action $\alpha\colon S_c\curvearrowright S/_\sim$ determines an action $\beta\colon S_c \curvearrowright \Gamma$ by graph automorphisms. In particular, $\beta_a(\Gamma_i)$ is a coconnected component of $\Gamma$ isomorphic with $\Gamma_i$ for all $i \in I$ and $a \in S_c$.
\end{proposition}
\begin{proof}
By Lemma~\ref{lem:alpha on I(S)}, $\alpha$ restricts to an action on the vertex set $\CI(S)/_\sim$. Recall that for $r,s,t \in S$ we have $s\perp t$ if and only if $rs \perp rt$ (using left cancellation). In particular, $([s],[t]) \in E$ is equivalent to $([as],[at]) \in E$ for all $s,t \in \CI(S), a \in S_c$. Therefore $\alpha_a$ induces an automorphism $\beta_a$ of the graph $\Gamma$. For every such map, $\beta_a(\Gamma_i)$ is again a coconnected component of $\Gamma$ as these are determined by the connectivity inside the graph $\Gamma$.
\end{proof}

As an immediate consequence of Proposition~\ref{prop:alpha restricts to co-conn comp}, we note:

\begin{corollary}\label{cor:alpha permutes co-conn comps}
If the coconnected components of $\Gamma$ are mutually nonisomorphic, then $\alpha$ restricts to an action $\alpha_a\colon S_c \curvearrowright V_i$ on the vertex sets of the coconnected components $\Gamma_i$ for all $i \in I$. 
\end{corollary}



\begin{definition}\label{def:balanced factorization}
Let $S$ be a right LCM monoid.
\begin{enumerate}[(i)]
\item The semigroup $S$ is \emph{noncore factorable} if every element in $S\setminus S_c$ is core equivalent to a finite product of noncore irreducibles.
\item The semigroup $S$ has \emph{balanced factorization} if $i(s)\neq i(t)$ implies $sS \cap tS=st'S, st'=ts'$ for some $s',t' \in \CI(S)$ with $i(s')=i(s),i(t')=i(t)$ for all $s,t \in \CI(S)$, where $i(s) \in I$ is the index of the coconnected component of $\Gamma$ containing $[s]$.
\end{enumerate}
\end{definition}

Our next target is to detect the existence of and compute a right LCM for two finite products of noncore irreducibles.

\begin{example}\label{ex:right LCM I(S) algorithm}
Assume that balanced factorization holds for $S$ and that the coconnected components of $\Gamma$ are edge-free. Let $s_1,s_2,t_1,t_2,t_3 \in \CI(S)$ and consider $s:=s_1s_2,t:=t_1t_2t_3$. Suppose for convenience that we already know that $s\Cap t$ holds (otherwise the process would terminate at one of the steps to come, and thus signalling $s\perp t$). Then we must have $s_1\Cap t_1$ because $sS \subset s_1S$ and $tS \subset t_1S$. Suppose $[s_1]$ and $[t_1]$ belong to the same coconnected component of $\Gamma$. As the component is edge-free, this forces $s_1\sim t_1$, so there are $s_1^{(1)},t_1^{(1)} \in S_c$ with $s_1S\cap t_1S= s_1t_1^{(1)}S, s_1t_1^{(1)}=t_1s_1^{(1)}$. Next, we will determine a right LCM for $s_1^{(1)}$ and $t_2$: Since $s_1^{(1)}\in S_c$, Lemma~\ref{lem:cap with core} implies that there are $s_1^{(2)} \in S_c$ and $t_2^{(1)} \in S$ with $s_1^{(1)}S\cap t_2S= s_1^{(1)}t_2^{(1)}S, s_1^{(1)}t_2^{(1)}=t_2s_1^{(2)}$. Due to Lemma~\ref{lem:alpha on I(S)}, we also know that $t_2^{(1)} \in \CI(S)$. A repetition of this argument yields a right LCM for $s_1^{(2)}$ and $t_3$: $s_1^{(3)} \in S_c$ and $t_3^{(1)} \in \CI(S)$ with $s_1^{(2)}S\cap t_3S= s_1^{(2)}t_3^{(1)}S, s_1^{(2)}t_3^{(1)}=t_3s_1^{(3)}$. This completes the first column in our diagram, where solid arrows refer to $\CI(S)$ and dotted ones to $S_c$, as in the proof of Lemma~\ref{lem:alpha on I(S)}:
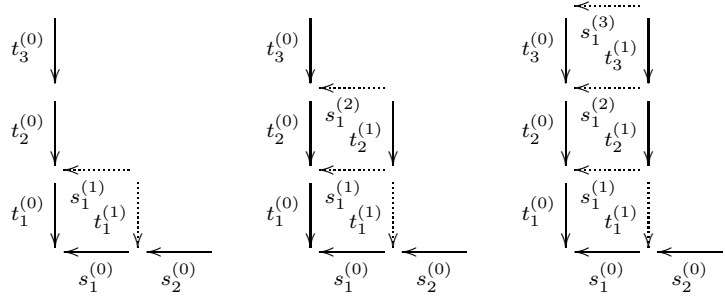
\[\begin{array}{ccc}
\xymatrix{
\ar_{t_3^{(0)}}[d] \\
\ar_{t_2^{(0)}}[d] \\
\ar_{t_1^{(0)}}[d] & \ar@{.>}^(.6){s_1^{(1)}}[l] \ar@{.>}_(.6){t_1^{(1)}}[d]\\
&\ar^{s_1^{(0)}}[l]&\ar^{s_2^{(0)}}[l]
}&
\xymatrix{
\ar_{t_3^{(0)}}[d] \\
\ar_{t_2^{(0)}}[d] & \ar@{.>}^(.6){s_1^{(2)}}[l] \ar_(.6){t_2^{(1)}}[d]\\
\ar_{t_1^{(0)}}[d] & \ar@{.>}^(.6){s_1^{(1)}}[l] \ar@{.>}_(.6){t_1^{(1)}}[d]\\
									& \ar^{s_1^{(0)}}[l]									    & \ar^{s_2^{(0)}}[l]
}&
\xymatrix{
\ar_{t_3^{(0)}}[d] & \ar@{.>}^(.6){s_1^{(3)}}[l] \ar_(.6){t_3^{(1)}}[d]\\
\ar_{t_2^{(0)}}[d] & \ar@{.>}^(.6){s_1^{(2)}}[l] \ar_(.6){t_2^{(1)}}[d] \\
\ar_{t_1^{(0)}}[d] & \ar@{.>}^(.6){s_1^{(1)}}[l] \ar@{.>}_(.6){t_1^{(1)}}[d]\\
									&\ar^{s_1^{(0)}}[l] 									   & \ar^{s_2^{(0)}}[l]
}\end{array}\]
For the bottom square in the second column, the same argument yields a right LCM for $s_2$ and $t_1^{(1)}$: $s_2^{(1)} \in \CI(S)$ and $t_1^{(2)} \in S_c$ with $s_2S\cap t_1^{(1)}S= s_2t_1^{(2)}S, s_2t_1^{(2)}=t_1^{(1)}s_2^{(1)}$. Thus we see that the encounter of a square with $s_{k+1}^{(\ell)} \sim t_{\ell+1}^{(k)}$ simplifies the task tremendously because we can then easily complete the corresponding row and column. Next, we focus on a right LCM for $s_2^{(1)}$ and $t_2^{(1)}$: The assumption $s\Cap t$ implies $s_2^{(1)} \Cap t_2^{(1)}$. Suppose that $i(s_2^{(1)}) \neq i(t_2^{(1)})$, that is, they belong to distinct coconnected components of $\Gamma$. Then balanced factorization grants $s_2^{(2)},t_2^{(2)} \in \CI(S)$ (with $i(s_2^{(1)}) = i(s_2^{(2)})$ and $i(t_2^{(1)}) = i(t_2^{(2)})$) such that $s_2^{(1)}S\cap t_2^{(1)}S= s_2^{(1)}t_2^{(2)}S, s_2^{(1)}t_2^{(2)}=t_2^{(1)}s_2^{(2)}$. For the right LCM for $s_2^{(2)}$ and $t_3^{(1)}$ suppose again that they have the same coconnected component, and thus $s_2^{(2)} \sim t_3^{(1)}$. This leads to the existence of $s_2^{(3)},t_3^{(2)} \in S_c$ with $s_2^{(2)}S\cap t_3^{(1)}S= s_2^{(2)}t_3^{(2)}S, s_2^{(2)}t_3^{(2)}=t_3^{(1)}s_2^{(3)}$, and our diagram is complete:
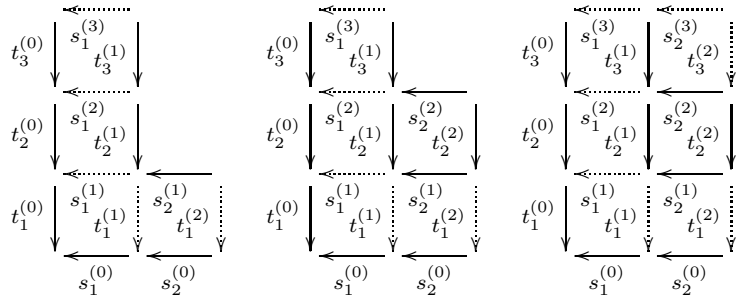
\[\begin{array}{ccc}
\xymatrix{
\ar_{t_3^{(0)}}[d] & \ar@{.>}^(.6){s_1^{(3)}}[l] \ar_(.6){t_3^{(1)}}[d] & \\
\ar_{t_2^{(0)}}[d] & \ar@{.>}^(.6){s_1^{(2)}}[l] \ar_(.6){t_2^{(1)}}[d] \\
\ar_{t_1^{(0)}}[d] & \ar@{.>}^(.6){s_1^{(1)}}[l] \ar@{.>}_(.6){t_1^{(1)}}[d] & \ar^(.6){s_2^{(1)}}[l] \ar@{.>}_(.6){t_1^{(2)}}[d]\\
									& \ar^{s_1^{(0)}}[l]									    & \ar^{s_2^{(0)}}[l]
}&
\xymatrix{
\ar_{t_3^{(0)}}[d] & \ar@{.>}^(.6){s_1^{(3)}}[l] \ar_(.6){t_3^{(1)}}[d] & \\
\ar_{t_2^{(0)}}[d] & \ar@{.>}^(.6){s_1^{(2)}}[l] \ar_(.6){t_2^{(1)}}[d]  & \ar^(.6){s_2^{(2)}}[l] \ar_(.6){t_2^{(2)}}[d]\\
\ar_{t_1^{(0)}}[d] & \ar@{.>}^(.6){s_1^{(1)}}[l] \ar@{.>}_(.6){t_1^{(1)}}[d] & \ar^(.6){s_2^{(1)}}[l] \ar@{.>}_(.6){t_1^{(2)}}[d]\\
									& \ar^{s_1^{(0)}}[l]									    & \ar^{s_2^{(0)}}[l]
}&
\xymatrix{
\ar_{t_3^{(0)}}[d] & \ar@{.>}^(.6){s_1^{(3)}}[l] \ar_(.6){t_3^{(1)}}[d] & \ar@{.>}^(.6){s_2^{(3)}}[l] \ar@{.>}_(.6){t_3^{(2)}}[d] \\
\ar_{t_2^{(0)}}[d] & \ar@{.>}^(.6){s_1^{(2)}}[l] \ar_(.6){t_2^{(1)}}[d] & \ar^(.6){s_2^{(2)}}[l] \ar_(.6){t_2^{(2)}}[d]\\
\ar_{t_1^{(0)}}[d] & \ar@{.>}^(.6){s_1^{(1)}}[l] \ar@{.>}_(.6){t_1^{(1)}}[d] & \ar^(.6){s_2^{(1)}}[l] \ar@{.>}_(.6){t_1^{(2)}}[d]\\
									& \ar^{s_1^{(0)}}[l]									    & \ar^{s_2^{(0)}}[l]
}
\end{array}\]
A right LCM for $s$ and $t$ is now given by $st_1^{(2)}t_2^{(2)}t_3^{(2)}$ (or taking the product along any path from the upper right to the lower left corner in the completed diagram).
\end{example}

\begin{remark}\label{rem:rLCM for products built from I(S)}
Let $s_1,\ldots,s_m,t_1,\ldots,t_n \in \CI(S)$ and $s:= s_1\cdots s_m, t:=t_1\cdots t_n$. Assume that balanced factorization holds for $S$ and that the coconnected components of $\Gamma$ are edge-free. We want to describe an algorithm that decides whether $s\Cap t$ or not, and produces a right LCM for $s$ and $t$ in case $s\Cap t$. Let us first fix the notation: We introduce two indices $k,\ell$ with initial value $(k,\ell) := (0,0)$ and range $0\leq k\leq m-1, 0\leq \ell \leq n-1$, and set $s_j^{(0)} := s_j$ for $1\leq j \leq m$, $t_j^{(0)} := t_j$ for $1\leq j\leq n$. In order to find a right LCM for $s$ and $t$, we need to find elements $s_{k+1}^{(\ell)},t_{\ell+1}^{(k)} \in \CI(S)\cup S_c$, where $(0,0)\leq (k,\ell) \leq (m-1,n-1)$, such that 
\begin{equation}\label{eq:intersection step}
s^{(\ell)}_{k+1}S \cap t^{(k)}_{\ell+1}S = s^{(\ell)}_{k+1}t^{(k+1)}_{\ell+1}S, \quad s^{(\ell)}_{k+1}t^{(k+1)}_{\ell+1}=t^{(k)}_{\ell+1}s^{(\ell+1)}_{k+1}
\end{equation}
holds for all $(0,0)\leq (k,\ell) \leq (m-1,n-1)$. In other words, we successively compute right LCMs for the (altered) factors of $s$ and $t$, and compose these to a right LCM of the product. We find it convenient to employ a flowchart for describing this repetitive process:

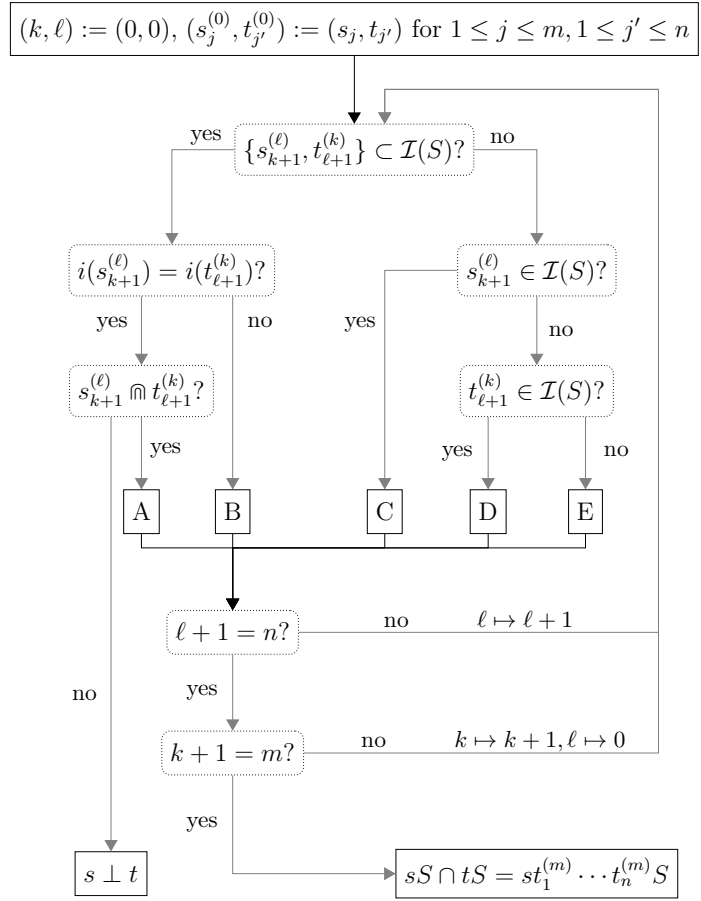
\pgfdeclarelayer{marx}
\pgfsetlayers{main,marx}
\providecommand{\cmark}[2][]{%
  \begin{pgfonlayer}{marx}
    \node [nmark] at (c#2#1) {#2};
  \end{pgfonlayer}{marx}
  } 
\providecommand{\cmark}[2][]{\relax} 
\begin{center}
\scalebox{.8}{
\begin{tikzpicture}[%
    >=triangle 60,              
    start chain=going below,    
    node distance= 20mm and 30mm,
    every join/.style={norm},   
    ]
\tikzset{
  base/.style={draw, on chain, on grid, align=center, minimum height=4ex}, 
  proc/.style={base, rectangle},
  test/.style={proc, densely dotted, rounded corners},
  term/.style={base, diamond, aspect=2, text width=10em},
  coord/.style={coordinate, on chain, on grid, node distance=6mm and 25mm},
  nmark/.style={draw, cyan, circle, font={\sffamily\bfseries}},
  norm/.style={->, draw, black},
  pos/.style={->,  black},
  neg/.style={->, draw, gray},
  it/.style={font={\small\itshape}}
}
\node [proc] (p0) {$(k,\ell):=(0,0)$, $(s^{(0)}_{j},t^{(0)}_{j'}):= (s_j,t_{j'})$ for $1\leq j \leq m, 1\leq j'\leq n$};
\node [test, below= of p0] (t1) {$\{s^{(\ell)}_{k+1},t^{(k)}_{\ell+1}\} \subset \CI(S)$?};
\node[test,below left=of t1] (t2) {$i(s^{(\ell)}_{k+1}) = i(t^{(k)}_{\ell+1})$?};
\node [test, below right= of t1] (t3) {$s^{(\ell)}_{k+1} \in \CI(S)$?};
\node[test,below left= 20mm and 5mm of t2] (t4) {$s^{(\ell)}_{k+1} \Cap t^{(k)}_{\ell+1}$?};
\node[test,below=20mm of t3] (t7) {$t^{(k)}_{\ell+1} \in \CI(S)$?};
\node [proc,below= 20mm of t4] (p1) {A};
\node [proc,right= 15mm of p1] (p2) {B};

\node [proc,right= 25mm of p2] (p3) {C};
\node [proc,below left= 20mm and 8mm of t7] (p4) {D};
\node [proc,below right= 20mm and 8mm of t7] (p7) {E};
\node [test,below= of p2] (t5) {$\ell+1=n$?};
\node [test,below= of t5] (t6) {$k+1=m$?};
\node [proc,below right= 100mm and 0mm of t3] (p5) {$sS\cap tS = st_1^{(m)}\cdots t_n^{(m)}S$};
\node [proc,below left= 100mm and 10mm of t2] (p6) {$s \perp t$};

\node [coord, left=of t1] (c1)  {}; 
\node [coord, right=of t1] (c2)  {}; 
\node [coord, left=of t3] (c3)  {};
\node [coord, below=of p1] (c41)  {};
\node [coord, below=of p2] (c42)  {};
\node [coord, below=of p3] (c43)  {};
\node [coord, below=of p4] (c44)  {};
\node [coord, below=of p7] (c45)  {};
\node [coord, left=of t4] (c5)  {}; 
\node [coord, left=20mm of p6]  (c6)  {}; 
\node [coord, right= 70mm of t5] (c7)  {};
\node [coord, above right= 10mm and 50mm of t1] (c8)  {};
\node [coord, above right= 10mm and 5mm of t1] (c9)  {};
\node [coord, right= 70mm of t6] (c10)  {};
\node [coord, below= 20mm of t6] (c11)  {};
\path (p0.south) to node [near start,xshift=-2mm, yshift=1mm] {} (t1);
  \draw [->,black] (p0.south) -- (t1);
\path (t1.west) to node [near start,xshift=-2mm, yshift=2mm] {\small{yes}} (c1); 
  \draw [->,gray] (t1.west) -- (c1) -| (t2);
\path (t1.east) to node [near start, xshift=2mm, yshift=2mm] {\small{no}} (c2); 
  \draw [->,gray] (t1.east) -- (c2) -| (t3);
\path (t2.south) to node [near start,xshift=-8mm, yshift=1mm] {\small{yes}} (t4); 
  \draw [->,gray] ([xshift=-5mm]t2.south) -- ([xshift=-5mm]t4);
\path (t2.south) to node [near start, xshift=10mm, yshift=11.7mm] {\small{no}} (p2); 
  \draw [->,gray] ([xshift=10mm]t2.south) -- (p2);
\path (t3.west) to node [near start,xshift=-10mm, yshift=-9mm] {\small{yes}} (c3); 
  \draw [->,gray] (t3.west) -- (c3) -- (p3);
\path (t3.south) to node [near start, xshift=4mm, yshift=0mm] {\small{no}} (t7); 
  \draw [->,gray] (t3.south) -- (t7);
\path (t7.south) to node [near start, xshift=-10mm, yshift=0mm] {\small{yes}} (p4); 
  \draw [->,gray] ([xshift=-8mm]t7.south) -- (p4);
\path (t7.south) to node [near start,xshift=10mm, yshift=0.4mm] {\small{no}} (p7); 
  \draw [->,gray] ([xshift=8mm]t7.south) -- ([xshift=0mm]p7.north);
\path (t4.south) to node [near start,xshift=-7mm, yshift=-10mm] {\small{no}} (p6); 
  \draw [->,gray] ([xshift=-5mm]t4.south) -- (p6);
\path (t4.south) to node [near start, xshift=4mm, yshift=0.5mm] {\small{yes}} (p1); 
  \draw [->,gray] ([xshift=0mm]t4.south) -- (p1);
\draw [->,black] (p1.south) -- (c41) -- (c42) -- (t5);
\draw [->,black] (p2.south) -- (c42) -- (t5);
\draw [->,black] (p3.south) -- (c43) -- (c42) -- (t5);
\draw [->,black] (p4.south) -- (c44) -- (c42) -- (t5);
\draw [->,black] (p7.south) -- (c45) -- (c42) -- (t5);
\path (t5.south) to node [near start,xshift=-5mm, yshift=0mm] {\small{yes}} (t6); 
  \draw [->,gray] (t5.south) -- (t6);
\path (t5.east) to node [near start, xshift=0mm, yshift=2mm] {\small{no} \hspace*{10mm}\small{$\ell \mapsto \ell+1$}} (c7); 
\draw [->,gray] (t5.east) -- (c7) -- (c8) -- (c9) -- ([xshift=5mm]t1.north);
\path (t6.south) to node [near start,xshift=-5mm, yshift=0mm] {\small{yes}} (c11); 
  \draw [->,gray] (t6.south) -- (c11) -- (p5);
\path (t6.east) to node [near start, xshift=2mm, yshift=2mm] {\small{no} \hspace*{10mm}\small{$k \mapsto k+1, \ell \mapsto 0$}} (c10); 
\draw [->,gray] (t6.east) -- (c10) -- (c8) -- (c9) -- ([xshift=5mm]t1.north);

\end{tikzpicture}
}
\end{center}

Before explaining the output, let us first describe the processes \boxed{A}--\boxed{E} from the flowchart above, in which we compute a right LCM (or rather the missing right factors) for $s^{(\ell)}_{k+1}$ and $t^{(k)}_{\ell+1}$. Apart from \boxed{E}, these did already make an appearance in Example~\ref{ex:right LCM I(S) algorithm}. We point out that $s^{(\ell)}_{k+1} \Cap t^{(k)}_{\ell+1}$ holds whenever we enter any of these processes within the algorithm. 
\begin{enumerate}
\item[\boxed{A}] We have $s^{(\ell)}_{k+1},t^{(k)}_{\ell+1} \in \CI(S)$ such that $[s^{(\ell)}_{k+1}]$ and $[t^{(k)}_{\ell+1}]$ belong to the same coconnected component of $\Gamma$, that is, $i(s^{(\ell)}_{k+1}) = i(t^{(k)}_{\ell+1})$. Then $s^{(\ell)}_{k+1} \Cap t^{(k)}_{\ell+1}$ and the assumption of edge-freeness force $s^{(\ell)}_{k+1} \sim t^{(k)}_{\ell+1}$, so Lemma~\ref{lem:rLCM if core eq} implies: 

$\rightsquigarrow$ There exist $s^{(\ell+1)}_{k+1}, t^{(k+1)}_{\ell+1} \in S_c$ with \eqref{eq:intersection step}.

\item[\boxed{B}] We have $s^{(\ell)}_{k+1},t^{(k)}_{\ell+1} \in \CI(S)$ and the two vertices $[s^{(\ell)}_{k+1}],[t^{(k)}_{\ell+1}]$ belong to distinct coconnected components of $\Gamma$, that is, $i(s^{(\ell)}_{k+1}) \neq i(t^{(k)}_{\ell+1})$. Here, balanced factorization implies:

$\rightsquigarrow$ There exist $s^{(\ell+1)}_{k+1}, t^{(k+1)}_{\ell+1} \in \CI(S)$ with \eqref{eq:intersection step} satisfying $i(s^{(\ell)}_{k+1}) = i(s^{(\ell+1)}_{k+1})$ and $i(t^{(k)}_{\ell+1}) = i(t^{(k+1)}_{\ell+1})$.

\item[\boxed{C}] We have $s^{(\ell)}_{k+1} \in \CI(S)$ and $t^{(k)}_{\ell+1} \in S_c$. Thus Lemma~\ref{lem:cap with core} and Lemma~\ref{lem:alpha on I(S)} imply:

$\rightsquigarrow$ There are $s^{(\ell+1)}_{k+1} \in \CI(S), t^{(k+1)}_{\ell+1} \in S_c$ with \eqref{eq:intersection step}.

\item[\boxed{D}] We have $s^{(\ell)}_{k+1} \in S_c$ and $t^{(k)}_{\ell+1} \in \CI(S)$. Thus Lemma~\ref{lem:cap with core} and Lemma~\ref{lem:alpha on I(S)} imply:

$\rightsquigarrow$ There are $s^{(\ell+1)}_{k+1} \in S_c, t^{(k+1)}_{\ell+1} \in \CI(S)$ with \eqref{eq:intersection step}.

\item[\boxed{E}] We have $s^{(\ell)}_{k+1},t^{(k)}_{\ell+1} \in S_c$. So Lemma~\ref{lem:cap with core} implies: 

$\rightsquigarrow$ There are $s^{(\ell+1)}_{k+1},t^{(k+1)}_{\ell+1} \in S_c$ with \eqref{eq:intersection step}.
\end{enumerate}
The algorithm starts with loop index $(k,\ell)=(0,0)$ and ends if it arrives at some level $(k',\ell')$ with $i(s_{k'+1}^{(\ell')})=i(t_{\ell'+1}^{(k')})$ and $s_{k'+1}^{(\ell')} \perp t_{\ell'+1}^{(k')}$, or if a right LCM has been obtained for $s_{m}^{(n-1)} \perp t_{n}^{(m-1)}$. In the first case, the algorithm shows
\[sS\cap tS \subset s_1\cdots s_{k'+1}S \cap t_1\cdots t_{\ell'+1}S = s_1\cdots s_{k'}t_1^{(k')}\cdots t_{\ell'}^{(k')}(s_{k'+1}^{(\ell')}S \cap t_{\ell'+1}^{(k')}S) = \emptyset,\]
which forces $s\perp t$. Similarly, the second case leads to
\[sS\cap tS = s_1\cdots s_m t_1^{(m)}\cdots t_n^{(m)}S = t_1\cdots t_n s_1^{(n)}\cdots s_m^{(n)}S,\]
in which case we obtain the right LCM for $s$ and $t$ by combining the right LCMs for the pairs appearing in the algorithm along a path on the grid from $(0,0)$ to $(m,n)$, compare Example~\ref{ex:right LCM I(S) algorithm}.
\end{remark}

\begin{lemma}\label{lem:permute in prod of irreds-new}
Suppose that all coconnected components of $\Gamma$ are edge-free and mutually nonisomorphic, and that $S$ has balanced factorization.

\begin{enumerate}[(i)]
\item Assume in addition that all coconnected components of $\Gamma$ are finite. If $s_1,\ldots,s_n \in \CI(S)$ and $\sigma$ is a permutation of $\{1,\ldots,n\}$, then there exist $t_1,\ldots,t_n \in \CI(S)$ with $s_1s_2\cdots s_n \sim t_1\cdots t_n$ and $i(t_k) = i(s_{\sigma(k)})$ for all $k$. 
\item Let $m,n \geq 1$. Whenever $s_1,\ldots,s_m,t_1,\ldots,t_n \in \CI(S)$ satisfy $s_1\cdots s_m \sim t_1\cdots t_n$, then $m=n$ and there is a permutation $\sigma$ of $\{1,\ldots,n\}$ such that $i(t_\ell) = i(s_{\sigma(\ell)})$ for every $1\leq \ell \leq n$.
\end{enumerate}
\end{lemma}
\begin{proof}
For (i), let $s_1,\ldots,s_n \in \CI(S), s:=s_1\cdots s_n$. Suppose that $\sigma_1,\sigma_2$ are permutations of $\{1,\ldots,n\}$ and $t_1^{(1)},\ldots,t_n^{(1)},t_1^{(2)},\ldots,t_n^{(2)} \in \CI(S)$ satisfy 
\begin{enumerate}[(1)]
\item $s\sim t_1^{(1)}\cdots t_n^{(1)} \text{ with } i(t_k^{(1)}) =i(s_{\sigma_1(k)}) \text{ for all }k$; and
\item $t_1^{(1)}\cdots t_n^{(1)} \sim t_1^{(2)}\cdots t_n^{(2)} \text{ with } i(t_k^{(2)}) =i(t_{\sigma_2(k)}^{(1)}) \text{ for all }k$.
\end{enumerate}
Then we deduce that $t_1^{(2)},\ldots,t_n^{(2)} \in \CI(S)$ satisfies $s\sim t_1^{(2)}\cdots t_n^{(2)} \text{ with } i(t_k^{(2)}) =i(t_{\sigma_2(k)}^{(1)}) = i(s_{\sigma_1\sigma_2(k)}) \text{ for all }k$. It therefore suffices to prove (i) for the Coxeter-Moore generators $\tau_k:=(k,k+1)$, where $1\leq k\leq n-1$, because every permutation is a finite product of these. So let $1\leq k \leq n-1$ and consider $\tau_k$ for $s:=s_1\cdots s_n$ with $s_1,\ldots,s_n \in \CI(S)$. If $i(s_k) =i(s_{k+1})$ holds, then $t_\ell:=s_\ell$ for all $\ell$ is a solution. In the case of $i(s_k) \neq i(s_{k+1})$, we start by choosing $t_j:=s_j$ for $1\leq j \leq k-1$. At step $k$, we need to invoke our assumptions: 

For every $t\in \CI(S)$ with $i(t)=i(s_{k+1})\neq i(s_k)$, there is an edge $([t],[s_k])$ in $\Gamma$, that is, $t\Cap s_k$. Thus balanced factorization yields elements $s_t,t' \in \CI(S)$ with $s_kS\cap tS = s_kt'S, s_kt'=ts_t$ and $i(s_t)=i(s_k), i(t')=i(t)=i(s_{k+1})$. For all $r,t \in \CI(S)$ with $i(t)=i(r)=i(s_{k+1})$ and $[t]\neq[r]$, edge-freeness of the coconnected components implies $t\perp r$. Therefore we get $s_kt' = ts_t \perp rs_r=s_kr'$, which is equivalent to $t'\perp r'$ by left cancellation. Hence $[t']$ and $[r']$ are distinct vertices from the coconnected component $\Gamma_{i(s_{k+1})}$. Since the vertex set of $\Gamma_{i(s_{k+1})}$ is finite, the map $t\mapsto t'$ is a bijection, so that there is $t\in \CI(S)$ with $i(t)=i(s_{k+1})$ and $[t']=[s_{k+1}]$. Every such $t$ satisfies $s_ks_{k+1} \sim s_kt'=ts_t$, say $s_ks_{k+1}S \cap ts_tS= s_ks_{k+1}aS, s_ks_{k+1}a=ts_tb$ for some $a,b \in S_c$. By setting $(t_k,t_{k+1}):= (t,s_tb)$ for such a $t$ which is fixed from now on (the class $[t]$ is uniquely determined), we have $s_ks_{k+1}a = t_kt_{k+1}$ for some $a \in S_c$ and $i(t_k)=i(s_{k+1})$. Since $t_{k+1}=s_tb \sim s_t$ and $i(s_t)=i(s_k)$, we also get $i(t_{k+1})=i(s_k)$. 

Using left cancellation, $t_j=s_j$ for $1\leq j\leq k-1$, and $t_kt_{k+1} = ts_tb = s_ks_{k+1}a$, we arrive at 
\[t_1\cdots t_{k+1}S \cap s_1\cdots s_nS = s_1\cdots s_{k+1}(aS\cap s_{k+2}\cdots s_nS).\]
Note that this also entails $t_1\cdots t_{k+1} \Cap s_1\cdots s_n$ due to $a\in S_c$. Since the coconnected components are mutually nonisomorphic, it follows from Corollary~\ref{cor:alpha permutes co-conn comps} that $aS\cap s_{k+2}S= at_{k+2}S, at_{k+2}=s_{k+2}a_{k+2}$ for some $a_{k+2} \in S_c$ and $t_{k+2} \in \CI(S)$ with $i(t_{k+2})=i(s_{k+2})$. Repeating this procedure for $(a_j,s_j)$ with $k+3\leq j \leq n$ allows us to arrive at a set of elements $t_{k+3},\ldots,t_n \in \CI(S)$ with $i(t_j)=i(s_j)$ for all $j\geq k+2$ and
\[\begin{array}{lcl}
t_1\cdots t_n &=& s_1\cdots s_{k+1}at_{k+2}t_{k+3}\cdots t_n \\
&=& s_1\cdots s_{k+1}s_{k+2}a_{k+2}t_{k+3}\cdots t_n \\
&\vdots&\\
&=& s_1\cdots s_na_n \\
&\sim& s_1\cdots s_n.
\end{array}\]
This completes the proof of (i) as $i(t_j) = i(s_{\tau_k(j)})$ holds for all $1\leq j \leq n$ with this choice of $t_1,\ldots,t_n\in \CI(S)$.

For (ii), let $s:=s_1\cdots s_m$ and $t:=t_1\cdots t_n$ and suppose $s\sim t$. Then $sS\cap tS = saS, sa=tb$ for some $a,b \in S_c$, see Lemma~\ref{lem:rLCM if core eq}. Thus Remark~\ref{rem:rLCM for products built from I(S)} states that $sS\cap tS= st_1^{(m)}\cdots t_n^{(m)}S = saS$, which forces $t_1^{(m)}\cdots t_n^{(m)} \in S_c$. As $S_c$ is hereditary, see Lemma~\ref{lem:rLCM if core eq}, we get $t_\ell^{(m)} \in S_c$ for every $1\leq \ell \leq n$. Recall from Remark~\ref{rem:rLCM for products built from I(S)} that $t_\ell^{(0)}=t_\ell \in \CI(S) \subset S\setminus S_c$, and that $t_\ell^{(k)} \in S_c$ implies $t_\ell^{(k')} \in S_c$ for all $k \leq k' \leq m$ since the flowchart will always end up in process \boxed{C} or \boxed{E} for these cases. It follows that, for each $1\leq \ell \leq n$, there is a unique minimal $k_\ell\geq 1$ such that $t_\ell^{(k_\ell)} \in S_c$, but $t_\ell^{(k)} \notin S_c$ for $0\leq k<k_\ell$. 

We will now argue that this entails a bijection $\{1,\ldots,m\} \to \{1,\ldots,n\}$: According to Remark~\ref{rem:rLCM for products built from I(S)}, the only process that leads to an output $t_\ell^{(k_\ell)}$ in the core starting from $t_\ell^{(k_\ell-1)} \in \CI(S)$ is \boxed{A}. Thus we must have $t_\ell^{(k_\ell-1)} \sim s_{k_\ell}^{(\ell-1)}$ for all $\ell$. In particular, this implies that 
\[s_{k_\ell}^{(\ell-1)} \in \CI(S) \quad \text{with } i(t_\ell^{(k_\ell-1)}) = i(s_{k_\ell}^{(\ell-1)}).\] Moreover, it follows that the map $\sigma\colon\{1,\ldots,m\} \to \{1,\ldots,n\}, \ell\mapsto k_\ell$ is injective: For every $\ell' > \ell$, we have $s_{k_\ell}^{(\ell'-1)} \in S_c$ because of $s_{k_\ell}^{(\ell)} \in S_c$. Therefore, the right LCM of $s_{k_\ell}^{(\ell'-1)}$ and $t_{\ell'}^{(k_\ell-1)}$ will be obtained either through process \boxed{D} or \boxed{E}. In the first case we get  $t_{\ell'}^{(k_\ell)}\notin S_c$, while the latter case requires $t_{\ell'}^{(k_\ell-1)} \in S_c$, which forces $k_{\ell'}< k_\ell$ due to minimality $k_{\ell'}$. 

Switching the role of the $s_i$ and the $t_j$, we get an injection $\sigma'\colon\{1,\ldots,n\} \to \{1,\ldots,m\}$, $k\mapsto \ell_k$. Together with the injective map $\sigma$, this shows $m=n$. Moreover, $s_{k_\ell}^{(\ell-1)} \in \CI(S)$ satisfies $i(t_\ell^{(k_\ell-1)}) = i(s_{k_\ell}^{(\ell-1)})$, allowing us to deduce $\ell_{k_\ell} = \ell$, that is, $\sigma'\circ \sigma = \id$.

Next, let us note that
\begin{enumerate}[(a)]
\item process \boxed{B} preserves the coconnected components due to balanced factorization, that is, $i(s^{(\ell)}_{k+1}) = i(s^{(\ell+1)}_{k+1})$ and $i(t^{(k)}_{\ell+1}) = i(t^{(k+1)}_{\ell+1})$; and that
\item processes \boxed{C} and \boxed{D} preserve the coconnected component of the respective noncore irreducible element due to mutually nonisomorphic coconnected components, see Corollary~\ref{cor:alpha permutes co-conn comps}.
\end{enumerate}
Since only the processes \boxed{B},\boxed{C}, or \boxed{D} have occured for computing the right LCM of $s_k^{(\ell-1)}$ and $t_\ell^{(k-1)}$ with $k< k_\ell$, and then \boxed{A} takes place for $k=k_\ell-1$, we deduce from (a) and (b) that 
\[ i(t_\ell) = i(t_\ell^{(k_\ell-1)}) = i(s_{k_\ell}^{(\ell-1)}) = i(s_{k_\ell})= i(s_{\sigma(\ell)}) \quad\text{holds for all } 1\leq \ell\leq n.\]
\end{proof}

\begin{thm}\label{thm:char existence of the GS}
A right LCM monoid $S$ admits a generalized scale if and only if the following conditions hold:
\begin{enumerate}[(i)]
\item The family $(\lvert V_i\rvert)_{i \in I} \subset \N^\times \cup \{\infty\}$ freely generates a nontrivial submonoid of $\N^\times$.
\item The coconnected components $\Gamma_i$ are edge-free.
\item $S$ is noncore factorable.
\item $S$ has balanced factorization.
\end{enumerate}
If $S$ satisfies (i)--(iv), then the generalized scale $N\colon S \to \N^\times$ is determined by $N_s := \lvert V_i \rvert$ for $[s] \in V_i$.
\end{thm}
\begin{proof}
Suppose first that there is a generalized scale $N$ on $S$. For $s,t \in \CI(S)$ with $s \not\sim t$, Proposition~\ref{prop:ABLS3.6}(ii) shows $([s],[t]) \notin E$ whenever $N_s=N_t$. Thus the coconnected component $\Gamma_i=(V_i,E_i)$ with $[s] \in V_i$ satisfies $N^{-1}(N_s)/_\sim \subset V_i$. On the other hand, if $N_s \neq N_t$, then Lemma~\ref{lem:s Cap t for i(s) neq i(t)} shows $s \Cap t$, that is, $([s],[t]) \in E$, because $N_s,N_t \in \text{Irr}~N(S)$ by Proposition~\ref{prop:indec = preimages of Irr(N(S))}. Thus we see that the coconnected components of $\Gamma(S)$ are given by $((N^{-1}(n)/_\sim,\emptyset))_{n \in \text{Irr}~N(S)}$. This shows (ii) and also (i), as $\text{Irr}~N(S)$ freely generates $N(S)$ by Proposition~\ref{prop:A3 gives A4}.

Property (iii) is a direct consequence of Lemma~\ref{lem:equiv prod of nc irreds}. Concerning (iv), if $s,t \in \CI(S)$ satisfy $N_s \neq N_t$, then $s\Cap t$ by Proposition~\ref{prop:indec = preimages of Irr(N(S))} and Lemma~\ref{lem:s Cap t for i(s) neq i(t)}. Due to Proposition~\ref{prop:ABLS3.6}(iv) and Proposition~\ref{prop:A3 gives A4}, we have $sS\cap tS = st'S, st'=ts'$ for some $s',t' \in S$ satisfying $N_{st'} = N_sN_t$. This yields $N_{s'}=N_s$ and $N_{t'}=N_t$, and hence (iv).

Conversely, suppose that $S$ satisfies properties (i)--(iv). We first note that (i) forces $\lvert V_i\rvert \geq 2$ for all $i \in I$ and that these are all mutually distinct. 
Using (i) and (iii), we define $N\colon S \to \N^\times$ by $N|_{S_c} =1$ and $s\sim s_1\cdots s_m \mapsto \prod_{k=1}^m \lvert V_{i(s_k)}\rvert$, where $s_k \in \CI(S)$ for all $k$. If $s\sim t_1\cdots t_n$ with $t_\ell \in \CI(S)$ for all $\ell$, then Lemma~\ref{lem:permute in prod of irreds-new}(ii) implies that $m=n$ and that there is a permutation $\sigma$ of $\{1,\ldots,m\}$ such that $i(t_\ell) = i(s_{\sigma(\ell)})$ for all $1\leq \ell \leq m$. In particular, we get $\prod_{k=1}^m \lvert V_{i(s_k)}\rvert = \prod_{\ell=1}^n \lvert V_{i(t_\ell)}\rvert$, and hence $N$ is a well-defined map. It is then easy to see that it is in fact a homomorphism of monoids. 

Suppose $s,t' \in S$ satisfy $N_s=N_{t'}$, and let $N_s=n_1\cdots n_m, n_k=\lvert V_{i_k}\rvert$, which is the unique factorization by (i). Due to (iii), $s\sim s_1\cdots s_m$ and $t'\sim t_1'\cdots t_m'$ for some $s_k,t'_k \in \CI(S)$. In addition, there exists a permutation $\sigma$ of $\{1,\ldots,m\}$ such that $i(s_k)=i(t'_{\sigma(k)})$ for all $k$. Invoking Lemma~\ref{lem:permute in prod of irreds-new}(i), we find $t_1,\ldots,t_m \in \CI(S)$ with $i(t_k)=i(t'_{\sigma(k)})=i(s_k)$ for all $k$ and $t' \sim t_1'\cdots t_m' \sim t:= t_1\cdots t_m$. We can now apply the algorithm of Remark~\ref{rem:rLCM for products built from I(S)} to $s$ and $t$. For $m=4$ with $s\sim t$, the output diagram from Remark~\ref{rem:rLCM for products built from I(S)} would be
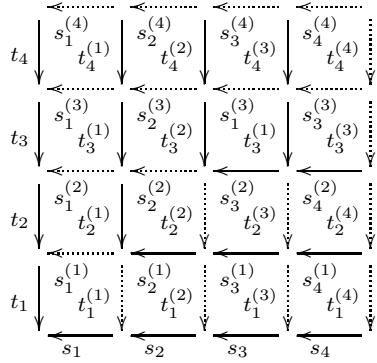
\[\scalebox{1}{
\xymatrix{
\ar_(.6){t_4}[d] & \ar@{.>}^(.6){s_1^{(4)}}[l] \ar_(.6){t_4^{(1)}}[d] & \ar@{.>}^(.6){s_2^{(4)}}[l] \ar_(.6){t_4^{(2)}}[d] & \ar@{.>}^(.6){s_3^{(4)}}[l] \ar_(.6){t_4^{(3)}}[d] & \ar@{.>}^(.6){s_4^{(4)}}[l] \ar@{.>}_(.6){t_4^{(4)}}[d]\\
\ar_(.6){t_3}[d] & \ar@{.>}^(.6){s_1^{(3)}}[l] \ar_(.6){t_3^{(1)}}[d] & \ar@{.>}^(.6){s_2^{(3)}}[l] \ar_(.6){t_3^{(2)}}[d] & \ar@{.>}^(.6){s_1^{(3)}}[l] \ar_(.6){t_3^{(1)}}[d] & \ar@{.>}^(.6){s_3^{(3)}}[l] \ar@{.>}_(.6){t_3^{(3)}}[d]\\
\ar_(.6){t_2}[d] & \ar@{.>}^(.6){s_1^{(2)}}[l] \ar_(.6){t_2^{(1)}}[d] & \ar@{.>}^(.6){s_2^{(2)}}[l] \ar@{.>}_(.6){t_2^{(2)}}[d] & \ar^(.6){s_3^{(2)}}[l] \ar@{.>}_(.6){t_2^{(3)}}[d] & \ar^(.6){s_4^{(2)}}[l] \ar@{.>}_(.6){t_2^{(4)}}[d]\\
\ar_(.6){t_1}[d] & \ar@{.>}^(.6){s_1^{(1)}}[l] \ar@{.>}_(.6){t_1^{(1)}}[d] & \ar^(.6){s_2^{(1)}}[l] \ar@{.>}_(.6){t_1^{(2)}}[d] & \ar^(.6){s_3^{(1)}}[l] \ar@{.>}_(.6){t_1^{(3)}}[d] & \ar^(.6){s_4^{(1)}}[l] \ar@{.>}_(.6){t_1^{(4)}}[d]\\
									& \ar^(.6){s_1}[l]& \ar^(.6){s_2}[l] & \ar^(.6){s_3}[l] & \ar^(.6){s_4}[l]
}}\]

where dashed arrows refer to $S_c$ while solid arrows refer to $\CI(S)$. More precisely, we get $s\perp t$, unless $s_1\sim t_1$ due to $i(s_1)=i(t_1)$ and (ii). This leads to $s_1^{(k)},t_1^{(k)} \in S_c$ for all $k\geq 1$. By iteration, we see that we only need to check that
\begin{equation}\label{eq:right LCM for prod of irred same length}
s_k^{(k-1)} \sim t_k^{(k-1)}
\end{equation}
holds for all $1\leq k \leq m$ to decide upon $s \Cap t$, as all other squares will involve (precisely) one element from $S_c$ in its lower half. This stems from the fact that if \eqref{eq:right LCM for prod of irred same length} holds for $k=1,\ldots,k_0$, then we get $s_k^{(\ell)},t_k^{(\ell)} \in S_c$ for all $1\leq k \leq \ell\leq m, k\leq k_0$. Now suppose there is $\ell\leq m$ such that \eqref{eq:right LCM for prod of irred same length} holds for all $k < \ell$. According to the diagram obtained from Remark~\ref{rem:rLCM for products built from I(S)} up to $\ell-1$, \eqref{eq:right LCM for prod of irred same length} will hold for $\ell$ if and only if 
\[ \alpha_{t_{\ell-1}^{(\ell-1)}\cdots t_1^{(\ell-1)}}^{-1}([s_\ell])=[s_\ell^{(\ell-1)}] = [t_\ell^{(\ell-1)}] = \alpha_{s_{\ell-1}^{(\ell-1)}\cdots s_1^{(\ell-1)}}^{-1}([t_\ell]).\]
For convenience, let $a:= t_{\ell-1}^{(\ell-1)}\cdots t_1^{(\ell-1)}, b:=s_{\ell-1}^{(\ell-1)}\cdots s_1^{(\ell-1)} (\in S_c)$. By (i), the coconnected components have vertex sets of mutually distinct cardinalities, so they are mutually nonisomorphic. Thus Corollary~\ref{cor:alpha permutes co-conn comps} applies, showing that $\alpha_a$ and $\alpha_b$ restrict to bijections on $V_{i(s_\ell)}$ and $V_{i(t_\ell)}$, respectively. In particular, we have $i(s_\ell^{(\ell-1)})=i(s_\ell)=i(t_\ell)=i(t_\ell^{(\ell-1)})$, so that $s_\ell^{(\ell-1)} \Cap t_\ell^{(\ell-1)}$ is in fact equivalent to $s_\ell^{(\ell-1)} \sim t_\ell^{(\ell-1)}$.

If \eqref{eq:right LCM for prod of irred same length} holds for all $k$, that is, $s\Cap t$, then the top row and the rightmost column of the diagram jointly mediate
\[s \sim s_1\cdots s_m \sim s_1\cdots s_mt_1^{(m)}\cdots t_m^{(m)} = ts_1^{(m)}\cdots s_m^{(m)} \sim t \sim t'.\]
Thus $N$ satisfies (ii) of Definition~\ref{def:gen scale}, that is, $N_s=N_{t'}$ implies $s\sim t'$ or $s\perp t'$.

A second conclusion we draw from the previous consideration is that $\lvert N^{-1}(n)/_\sim\rvert = n$ for all $n \in N(S)$: For $n=n_1\cdots n_m, n_k=\lvert V_{i_k}\rvert$, we can pick 
\[\{s_{k,j} \mid 1\leq k \leq m, 1\leq j\leq n_k\} \subset \CI(S)\] 
such that $i(s_{k,j}) = i_k$ and $s_{k,j} \perp s_{k,j'}$ for all $j\neq j'$. This already shows $\lvert N^{-1}(n)/_\sim\rvert \geq n$ as $(s_{1,j_1}\cdots s_{m,j_m})_{1\leq j_k \leq n_k} \subset N^{-1}(n)$ consists of $n$ mutually orthogonal elements. Now for $[t] \in N^{-1}(n)/_\sim$, we can assume that $t\sim t_1\cdots t_m$ with $i(t_k)=i_k$ for all $k$ (using the replacement technique for $t'\mapsto t\sim t'$ from before). We are going to find $(j_1,\ldots,j_m)$ such that $t\sim s_{1,j_1}\cdots s_{m,j_m}$: As $n_1=\lvert V_{i_1}\rvert$, there is a unique $1\leq j_1 \leq n_1$ such that $t_1 \sim s_{1,j_1}$. Therefore \eqref{eq:right LCM for prod of irred same length} holds for $k=1$ and we get $t_1^{(k)},s_{1,j_1}^{(k)} \in S_c$ for all $1\leq k \leq m$. As $\alpha_{t_1^{(k)}}\alpha_{s_{1,j_1}^{(1)}}^{-1}$ restricts to a bijection of $V_{i_2}$, see Corollary~\ref{cor:alpha permutes co-conn comps}, there is $1\leq j_2 \leq n_2$ such that $\alpha_{t_1^{(1)}}\alpha_{s_{1,j_1}^{(1)}}^{-1}([t_2]) = [s_{2,j_2}]$, which is \eqref{eq:right LCM for prod of irred same length} for $k=2$. Iterating this process, where at stage $k$ we find $1\leq j_k\leq n_k$ with
\[\alpha_{t_{k-1,j_{k-1}}^{(k-1)}\cdots t_{1,j_1}^{(1)}}\alpha_{s_{k-1,j_{k-1}}^{(k-1)}\cdots s_{1,j_1}^{(1)}}^{-1}([t_k]) = [s_{k,j_k}]\] 
that gives \eqref{eq:right LCM for prod of irred same length}. This shows $\lvert N^{-1}(n)/_\sim\rvert = n$. 

The above procedure is also useful in proving (iii) of Definition~\ref{def:gen scale}: Let $s \in S, n \in N(S)$. Without loss of generality, we may assume $s \in S\setminus S_c, n>1$ as the remaining cases are trivial due to the presence of core elements. Then (iii) asserts that $s\sim s_1\cdots s_{m_1}$ for some $s_k \in \CI(S)$ with $\lvert V_{i(s_k)}\rvert =: n_{1,k}$, while (i) gives $n= n_{2,1}\cdots n_{2,m_2}$ for uniquely determined $n_{2,k} \in \{\lvert V_i\rvert \mid i \in I\}$. By rearrangement of the factors $n_{2,k}$ and possibly switching $s\mapsto s' \sim s$ via Lemma~\ref{lem:permute in prod of irreds-new}(i), we can assume that there is $0 \leq \ell \leq m_1 \wedge m_2$ such that $n_{1,k}=n_{2,k}$ for $1\leq k \leq \ell$, while $n_{1,k}\neq n_{2,k'}$ for all $\ell < k,k'$. In plain words, we extract the greatest common divisor of $n_{1,1}\cdots n_{1,m_1}$ and $n$. We then simply set $t_k:= s_k$ for $1\leq k \leq \ell$ and fix an arbitrary choice of $t_{\ell+1},\ldots,t_{m_2}$ with $n_{2,k} = \lvert V_{i(t_k)}\rvert$ for $\ell+1\leq k\leq m_2$. Looking at Remark~\ref{rem:rLCM for products built from I(S)} for $s_{\ell+1}\cdots s_{m_1}$ and $t_{\ell+1},\ldots,t_{m_2}$, we see that due to $n_{1,k}\neq n_{2,k'}$ for all $\ell < k,k'$, we can always complete the diagram as we only apply process \boxed{B} in every step. The cause of this outcome is nothing but balanced factorization. Thus we find $t \in N^{-1}(n)$ with $s\Cap t$ as required, and we have shown that $N$ is indeed a generalized scale.
\end{proof}

\begin{remark}\label{rem:existence and uniqueness}
Due to the explicit description of the generalized scale in terms of $\Gamma(S)$, Theorem~\ref{thm:char existence of the GS} not only characterizes the existence of the generalized scale, but also implies uniqueness, which we already showed abstractly in Theorem~\ref{thm:uniqueness}.
\end{remark}

\section{Examples, challenges, and Saito's degree map}\label{sec:pushing borders}

\subsection{Self-similar group actions}\label{subsec:ex1}

Let $(G,X)$ be a self-similar group action, that is, $X$ is a finite alphabet in at least two letters and $G$ is a group acting on the free monoid $X^*$ generated by $X$ such that for each $x \in X, g \in G$, there are unique elements $g(x) \in X, g|_x \in G$ such that $g(xw) = g(x)g|_x(w)$ for all $w \in X^*$. The associated Zappa-Sz\'{e}p product $S:= X^*\bowtie G$ is the monoid with elements $X^*\times G$ and product $(v,g)(w,h) := (vg(w),g|_wh)$, where the expression $g|_w$ is defined recursively through $g|_{xu} := (g|_x)|_u$. It was observed that $S$ is a right LCM monoid in \cite{BRRW}*{Theorem~3.8}, and that $S$ admits a generalized scale, see \cite{ABLS}*{Subsection~5.3}. We can recover the second part easily from Theorem~\ref{thm:char existence of the GS}: 
\begin{enumerate}[(a)]
\item $S_c = S^*=\{ (\varnothing, g) \mid g \in G \} \cong G$.
\item An element $(w,g)$ belongs to $\CI(S)$ if and only if $w \in X$.
\item For every $(w,g) \in S\setminus S_c$, $w$ is a nontrivial word $w=x_1x_2\cdots x_n$ in $X$ of length $n\geq 1$. Thus $(w,g) = (x_1,1)(x_2,1)\cdots (x_{n-1},1)(x_n,g)$ shows that $S$ is noncore factorable.
\item For $(x,g),(y,h) \in \CI(S)$, we get $(x,g)\Cap (y,h)$ if and only if $x=y$ if and only if $(x,g) \sim (y,h)$. Thus, $\Gamma(S)$ is the empty graph on $\lvert X \rvert$ vertices $[(x,1)], x \in X$, which is coconnected.
\end{enumerate} 
Hence conditions (i)--(iv) in Theorem~\ref{thm:char existence of the GS} are satisfied and $(w,g) \mapsto \lvert X\rvert^{\ell(w)}$ defines a generalized scale, where $\ell\colon X^* \to \N$ denotes the word length with respect to $X$.

\subsection{\texorpdfstring{The ax$+$b-semigroup over the natural numbers}{The natural numbers}}\label{subsec:ex2}
A famous example of a right LCM monoid with a generalized scale is the ax$+$b-semigroup over the natural numbers $S:= \N\rtimes\N^\times$, see \cite{ABLS}*{Subsection~5.4}. This example features: 
\begin{enumerate}[(a)]
\item $S_c = \{ (n,1) \mid n \in \N\} \cong \N$.
\item An element $(m,p)$ belongs to $\CI(S)$ if and only if $p$ is a prime.
\item As every positive integer is a finite product of primes, $S$ is noncore factorable.
\item Let $(m,p),(n,q) \in \CI(S)$. If $p=q$, then $(m,p)\Cap (n,q)$ holds if and only if $m-n \in p\Z$, that is, $(m,p) \sim (n,q)$. If $p \neq q$, then $p\Z + q\Z = \text{gcd}(p,q)\Z =\Z$ and $(m,p)S \cap (n,q)S = (m+pn',pq)S, m+pn'=n+qm'$, where $m',n' \in \N$ represent the smallest nonnegative solution for $m +pn' = n+ qm'$ (the solution in $\Z$ is unique up to $pq\Z$). We thus get the coconnected components of $\Gamma(S)$ to be $(V_p)_{p \in P}$, where $P$ denotes the primes in $\N^\times$, with $V_p = \{[(k,p)] \mid 0\leq k <p\}$ and $E_p=\empty$.
\item In this example, balanced factorization mirrors the fact that the prime factorization mentioned in (c) is unique (up to permutation of factors).
\end{enumerate}
So Theorem~\ref{thm:char existence of the GS} implies that $(m,p) \mapsto p$ is the unique generalized scale on $S$. In passing, we note that Theorem~\ref{thm:char existence of the GS} and the notion of a generalized scale thereby provide a profound justification of ``the obvious choice'' made in \cite{LR2} as described in \cite{BLRS}*{Remark~2.4}.

\subsection{Algebraic dynamical systems}\label{subsec:ex3}
Let $(G,P,\theta)$ be an algebraic dynamical system as considered in \cite{BLS2}, that is, $P$ is a right LCM monoid acting upon a discrete group $G$ by injective group endomorphisms $\theta_p$ such that $pP\cap qP = rP$ implies $\theta_p(G) \cap \theta_q(G) = \theta_r(G)$. The right LCM monoid of interest here is $S := \gxp$, and we assume that $p \in P^*$ whenever $\theta_p \in \Aut(G)$\footnote{One may always achieve this by replacing $P$ by the acting semigroup of endomorphisms.}. 

\begin{remark}\label{rem:I(GxP)=GxI(P)}
The set $\CI(\gxp)$ is given by $G\times \CI(P)$. It is shown in \cite{ABLS}*{Proposition~5.11(i)} that $S$ admits a generalized scale (given by $N_{(g,p)} := [G:\theta_p(G)]$) provided that:
\begin{enumerate}[(a)]
\item The index $N_p:= [G:\theta_p(G)]$ is finite for all $p \in P$, and $N_p>1$ for some $p \in P$.
\item If $N_p=N_q$ for some $p,q \in P$, then $p \in qP^*$.
\end{enumerate}
\end{remark}

It is easy to see from Theorem~\ref{thm:char existence of the GS}(i) and (iii) that part (a) is also necessary: Suppose first that there was $p \in P$ with infinite index. For the semigroup $S=\gxp$ to be noncore factorable, $P$ has to be noncore factorable, so $p$ would have to be expressible as a finite product $p=p_1q_2\cdots p_n$ for some $p_k \in \CI(P)$. Since the index is multiplicative, there would exist $q \in \CI(P)$ with infinite index. We observe that $(g,q) \in \CI(S)$ for all $g \in G$ with $(g,q)\sim (h,q)$ if and only if $g^{-1}h \in \theta_q(G)$. Thus the coconnected component $\Gamma_i$ of $[(g,q)]$ will contain infinitely many vertices, thereby violating condition (i) of Theorem~\ref{thm:char existence of the GS}. So $P$ must act by finite index endomorphisms. If all of them are automorphisms, then $S$ is a group and hence $\Gamma = (\{v\},\emptyset)$ will be an obstruction to nontriviality in Theorem~\ref{thm:char existence of the GS}~(i). The next example shows that part (b) is too strong in general:

\begin{example}[Freely doubled $\times 2$]\label{ex:ads with free monoid}
Let $P$ be the free monoid generated by two elements $p,q$, $G=\Z$, and $\theta_p=\theta_q = \times 2\colon \Z \to \Z$. Then $S=\gxp$ is a right LCM monoid with $\CI(S) = \{ (n,r) \mid n \in \Z, r \in \{p,q\}\}$, see Remark~\ref{rem:I(GxP)=GxI(P)}. We note that $(m,p) \perp (n,q)$ for all $m,n \in \Z$, while $(m,r) \cap (n,r)$ if and only if $m-n \in \theta_r(G) = 2\Z$, that is, $(m,r) \sim (n,r)$, for $r \in \{p,q\}$. Hence $\Gamma(S)$ is the empty graph on four vertices $[(0,p)], [(1,p)],[(0,q)], [(1,q)]$. Therefore, $\Gamma(S)$ is coconnected and conditions (i),(ii), and (iv) of Theorem~\ref{thm:char existence of the GS} follow immediately, while (iii) boils down to Remark~\ref{rem:I(GxP)=GxI(P)}, noting that $p$ and $q$ generate $P$ by construction. Thus $S$ admits a generalized scale $N$ with $N_{(g,w)} = 4^{\ell(w)}$, where $\ell$ denotes the word length on $P$ with respect to $\{p,q\}$.  
\end{example}

On the other hand, $\gxp$ may fail to have a generalized scale for $P\cong \N^2$ as soon as condition (b) from Remark~\ref{rem:I(GxP)=GxI(P)} is violated:

\begin{example}[Ledrappier's shift]\label{ex:Ledrappier shift}
Consider the right LCM monoid $S:= G \rtimes_{\sigma,\id+\sigma} \N^2$, where $G:= \bigoplus_\N \Z/2\Z$ and $\sigma$ is the right shift $(g_0,g_1\ldots) \mapsto (0,g_0,g_1,\ldots)$. In other words, $S$ is the semidirect product of the standard restricted wreath product $\Z/2\Z \wr \N$ by $\N$, where the latter action is given by $\id+\sigma$. This monoid has close ties to the well-known Ledrappier shift from \cite{Led}, see \cite{Sta2}*{Example~2.9} and \cite{ER}*{Section~11}. It is also an example of a right LCM monoid arising from an algebraic dynamical system in the sense of \cite{BLS2}. The two generators for the $\N^2$-action $\sigma$ and $\id+\sigma$ are independent\footnote{This strong form of commutativity for injective group endomorphisms of discrete groups surfaced in \cite{CuntzVershik}, and was analysed more closely in \cite{Sta1}*{Section~1} and \cite{Sta2}*{Section~2}, exhibiting a close connection to $*$-commutativity from \cite{AR}.} injective group endomorphisms whose images have index two. More precisely, we have $G = \sigma(G) \sqcup e_0+\sigma(G) = (\id+\sigma)(G) \sqcup e_0+(\id+\sigma)(G)$, where $e_k := (\delta_{k,n})_{n \in \N}$. This leads us to the following features:
\begin{enumerate}[(a)]
\item $S_c = S^* = G\times \{0\}$.
\item An element $(g,m)$ belongs to $\CI(S)$ if and only if $m \in \{\sigma,\id+\sigma\}$.
\item The semigroup $S$ is noncore factorable because $\langle \sigma,\id+\sigma\rangle \cong \N^2$, see Remark~\ref{rem:I(GxP)=GxI(P)}.
\item Let $(g,m),(h,n) \in \CI(S)$. If $m=n$, then $(g,m)\Cap (h,n)$ if and only if $g_0=h_0$ if and only if $(g,m)\sim (h,n)$. For $m\neq n$, we always get $(g,m)\Cap (h,n)$ with $(g,m)S\cap (h,n)S = (gm(h'),mn)S, gm(h') = hn(g')$ for suitable $g',h' \in G$ due to strong independence of $\sigma$ and $\id+\sigma$, that is, $\sigma(G)+(\id+\sigma)(G) = G$, see \cite{Sta1}*{Definition~1.3}. This means, $\Gamma(S)$ has two coconnected components $\Gamma_i,i=1,2$ given by $V_1 := \{[(g,\sigma)] \mid g \in \{0,e_0\}\}$ and $V_2 := \{[(g,\id +\sigma)] \mid g \in \{0,e_0\}\}$ with $E_i = \emptyset$.
\item Balanced factorization holds because of independence, see (d).
\end{enumerate}
We note that unlike for the previous examples from \ref{subsec:ex1},\ref{subsec:ex2}, and Example~\ref{ex:Z2 xp with flip}, Corollary~\ref{cor:alpha permutes co-conn comps} does not apply, but its conclusion holds nonetheless. The natural alternative to a generalized scale for $S$ is the homomorphism given by $g \mapsto 1$ for $g \in G$ and $\sigma \mapsto 2, \id+\sigma \mapsto 2$. We expect the corresponding dynamics on $C^*(S)$ to have a KMS-state structure of the same kind as we have established for right LCM monoids admitting a generalized scale, see \cite{ABLS}*{Theorem~4.3} and \cite{BLRSt}.
\end{example}

While all the previous examples satisfy the conclusion of Corollary~\ref{cor:alpha permutes co-conn comps} that $\alpha\colon S_c \curvearrowright S/_\sim$ restricts to permutations on the vertex sets of the coconnected components of $\Gamma(S)$, the following is an easy example, where this is no longer the case. As this requires the existence of two distinct isomorphic coconnected components in $\Gamma(S)$, any such example will fail condition~(i) of Theorem~\ref{thm:char existence of the GS}. It will thus not have a generalized scale, which indicates that despite the usefulness of the notion of generalized scales, we ought to consider more flexible analogues eventually.

\begin{example}[Multiplication with flip]\label{ex:Z2 xp with flip}
Let $p\in \Z, |p| \geq 2$ and 
\[p_0 := \begin{pmatrix}p&0\\0&1\end{pmatrix}, \ p_1 := \begin{pmatrix}1&0\\0&p\end{pmatrix}, \ x := \begin{pmatrix}0&1\\1&0\end{pmatrix} \in M_2(\Z).\] 
Consider $S := \Z^2 \rtimes P$, where $P\subset M_2(\Z)$ is the monoid generated by $p_0$ and $x$. Then $S$ is a right LCM monoid and $P = \langle p_0,p_1\rangle \rtimes \langle x\rangle \cong \N^2\rtimes \Z/2\Z$ with $x p_0 x = p_1$. We obtain:
\begin{enumerate}[(a)]
\item $S_c = S^* = \Z^2 \rtimes \Z/2\Z$.
\item An element $(m,q)$ belongs to $\CI(S)$ if and only if $q = p_ix^\ell$ for $i,\ell \in \{0,1\}$ as every element in $P$ admits a unique normal form $q = p_{k_1}p_{k_2}\cdots p_{k_n}x^\ell$ with $k_j,\ell \in \{0,1\}$.
\item The normal form described in (b) implies that $S$ is noncore factorable.
\item Let $(m,q_0),(n,q_1) \in \CI(S)$. If $q_i=p_kx^{\ell_i}$ for some $k,\ell_i\in \{0,1\}$, then $(m,q_0)\Cap (n,q_1)$ holds if and only if $m-n \in p_k\Z^2$, that is, $(m,q_0) \sim (n,q_1)$. On the other hand, if $q_i=p_ix^{\ell_i}$ for some $\ell_i\in \{0,1\}, i=1,2$, then $(m,q_0)\Cap (n,q_1)$ always holds as $p_1\Z^2 + p_2\Z^2 = \Z^2$. Thus $\Gamma(S)$ is given by two coconnected components $\Gamma_0$ and $\Gamma_1$ with $V_i := \{[(m,p_0)] \mid 0\leq m < p e_i\}$ for $i=0,1$, where $e_0=(1,0),e_1=(0,1)$, and $E_i=\emptyset$.
\item Balanced factorization is due to $p_0\Z^2 \cap p_1\Z^2 = p_0p_1\Z^2, p_0p_1=p_1p_0$.
\end{enumerate} 
For $[((m_0,m_1),p_0)] \in V_0$, we get $\alpha_x([((m_0,m_1),p_0)]) = [((m_1,m_0),p_1)] \in V_1$, so the flip switches the two coconnected components $\Gamma_0$ and $\Gamma_1$.
\end{example}

\subsection{Graph products}\label{subsec:ex4}
Suppose $\Lambda = (W,F)$ is an undirected graph, and $(S_w)_{w \in W}$ is a family of right LCM monoids. The graph product $S(\Lambda,(S_w)_{w \in W})$ is then defined as the quotient of the free monoid $\ast_{w \in W} S_w$ by the congruence generated by the relations $st=ts$ for $s \in S_v,t \in S_w$ with $(v,w) \in F$. Whenever there is no ambiguity concerning the vertex semigroups $(S_w)_{w \in W}$, we shall simply write $S(\Lambda)$. This construction goes back to the case of groups first considered in \cite{Gre}, and was studied for monoids in \cites{Cos,FK}. Prominent examples of this kind are right-angled Artin monoids, which have already been of interest to operator algebraists, see for instance \cites{CrispLacaT,CrispLaca,ELR,Sta4}. We would like to mention that extensive research has been conducted on such monoids under the name of \emph{trace monoids} in connection theoretical computer science, see \cite{DR}. A standard means of studying the structure of graph products is the decomposition of the graph $\Lambda$ into its \emph{coconnected components} $\Lambda_i, i \in I$. For convenience, let us recall that graph $(V,E)$ is called coconnected if the graph $(V,V\times V\setminus E)$ is connected, see also \cite{Sta4}*{Subsection~2.2}. Let $(\Lambda_j)_{j \in J}$ be the decomposition into coconnected components for $\Lambda$, so that the graph product satisfies $S(\Lambda) = \bigoplus_{j \in J} S(\Lambda_j)$. Now for each $j \in J$, we let $(\Gamma_i)_{i \in I_j}$ with $\Gamma_i = (V_i,E_i)$ denote the decomposition of $\Gamma(S(\Lambda_j))$ into its coconnected components.

\begin{thm}\label{thm:gen scales for graph products}
Let $\Lambda = (W,F)$ be an undirected graph and $(S_w)_{w \in W}$ a family of right LCM monoids. Then: 
\begin{enumerate}[(i)]
\item $\CI(S(\Lambda)) = \bigsqcup_{j \in J} \CI(S(\Lambda_j)) \oplus \bigoplus_{j' \in J\setminus\{j\}} S(\Lambda_{j'})_c$.
\item The coconnected components of $\Gamma(S(\Lambda))$ are edge-free if and only if the coconnected components of $\Gamma(S(\Lambda_j))$ are edge-free for all $j \in J$.
\item $S(\Lambda)$ is noncore factorable if and only if $S(\Lambda_j)$ is noncore factorable for all $j \in J$.
\item $S(\Lambda)$ has balanced factorization if and only if $S(\Lambda_j)$ has balanced factorization and $\alpha\colon S(\Lambda_j)_c \curvearrowright S(\Lambda_j)/_\sim$ restricts to an action on the vertex sets of the individual coconnected components for all $j \in J$.
\end{enumerate}
In particular, the graph product $S(\Lambda)$ has a generalized scale if and only if the properties (ii)--(iv) from Theorem~\ref{thm:char existence of the GS} hold for every $S(\Lambda_j), j \in J$, and $(\lvert V_i\rvert)_{i \in \bigsqcup_{j \in J} I_j} \subset \N^\times \cup \{\infty\}$ freely generates a nontrivial submonoid of $\N^\times$.
\end{thm}
\begin{proof}
Part (i) is a straightforward observation using the definition of $S(\Lambda)$, and (iii) then follows easily. For $s = \sum_{j \in J} s_j \in \CI(S(\Lambda)$, we let $j(s) \in J$ be the unique index with $s_{j(s)} \in \CI(S(\Lambda_{j(s)}))$. Then $s,t \in \CI(S(\Lambda)$ satisfy $s\Cap t$ if and only if $j(s) \neq j(t)$ or $j(s)=j(t)$ and $s_{j(s)} \Cap t_{j(s)}$. Thus the coconnected components of $\Gamma(S(\Lambda))$ are given by $(\Gamma_i)_{i \in \bigsqcup_{j \in J} I_j}$. With this insight, (ii) is an immediate consequence. By (i), balanced factorization for every $j \in J$ is necessary, corresponding to the case of $s,t \in \CI(S(\Lambda))$ with $j(s) =j(t)$. For $j(s) \neq j(t)$, the description in (i) yields $sS(\Lambda) \cap tS(\Lambda) =st'S(\Lambda), st'=ts'$ with $t'_j \in S(\Lambda_j)_c$ if and only if $j\neq j(t)$ and $s'_j \in S(\Lambda_j)_c$ if and only if $j\neq j(s)$. At $j(t)$ we have $[t'_{j(t)}] = \alpha_{s_{j(t)}}^{-1} ([t_{j(t)}]) \in \CI(S(\Lambda_{j(t)}))$, and similarly for $s'$ at $j(s)$. Thus the necessary and sufficient condition for balanced factorization is obtained by adding the assumption that $\alpha\colon S(\Lambda_j)_c \curvearrowright S(\Lambda_j)/_\sim$ restricts to an action on the vertex set $V_i$ of each coconnected component $\Gamma_i, i \in I_j$, for all $j \in J$, compare Corollary~\ref{cor:alpha permutes co-conn comps}. This proves (iv), and the claim about the existence of generalized scales is now a consequence of Theorem~\ref{thm:char existence of the GS}.
\end{proof}

\begin{remark}\label{rem:cond (i) reinterpreted}
The condition that $(\lvert V_i\rvert)_{i \in \bigsqcup_{j \in J} I_j} \subset \N^\times \cup \{\infty\}$ freely generates a nontrivial submonoid of $\N^\times$ can also be expressed by saying that 
\begin{enumerate}[(a)]
\item $(\lvert V_i\rvert)_{i \in I_j} \subset \N^\times \cup \{\infty\}$ freely generates a submonoid $M_j$ of $\N^\times$ for all $j \in J$;
\item $M_j \cap M_{j'} = \{1\}$ for all $j,j' \in J, j \neq j'$; and 
\item there is $j \in J$ such that $M_j$ is nontrivial.
\end{enumerate}
In particular, we see that a necessary condition for $S(\Lambda)$ to admit a generalized scale $N$ is that $S(\Lambda_j)$ admits a generalized scale for all $j \in J$ with nontrivial $M_j$. In this case, $N$ is determined by this family of generalized scales.
\end{remark}

According to Theorem~\ref{thm:gen scales for graph products}, the study of generalized scales (or analogues thereof) on graph products can be reduced to the coconnected case. But the situation seems to be quite intricate, as the following example shows:

\begin{example}[Graph products behaving badly]\label{ex:graph products gone mad}
Let us consider the graph product for $\Lambda =(\{v,w\},\emptyset)$ with 
\[S_v = \langle b\rangle \rtimes \langle q\rangle \cong \N\rtimes \N \cong \langle b\rangle \rtimes \langle q\rangle = S_w\]
for $p,q \in \N_{\geq 2}$, that is, the free product of two one-dimensional subdynamics from \ref{subsec:ex2}. As $\Lambda$ is edge-free, $S:=S(\Lambda)$ has
\begin{enumerate}[(a)]
\item $S_c=S^*= \{1_S\}$; and
\item $\CI(S) = \{(a,1),(0,p),(b,1),(0,q)\}$ with $\Gamma(S)$ having two edges $([(a,1)],[(0,p)])$ and $([(b,1)],[(0,q)])$.
\end{enumerate}
So $\Gamma(S)$ is coconnected but has edges, so $S$ does not satisfy (ii) of Theorem~\ref{thm:char existence of the GS}. Therefore $S$ does not admit a generalized scale even though both $S_v$ and $S_w$ do.
\end{example}

It appears that a natural choice for the substitute of the generalized scale in Example~\ref{ex:graph products gone mad} would be obtained by combining the generalized scales on $S_v$ and $S_w$ suitably with regard to $\Lambda$. More precisely, we could consider the homomorphism determined by $(a,1),(b,1) \mapsto 2, (0,p) \mapsto 2p$, and $(0,q)\mapsto 2q$. In this specific example we might, for each coconnected component,
\begin{enumerate}[(a)]
\item count the number of connected components, and
\item pick the generalized scale on each of the connected components
\end{enumerate}
instead.

Such problems cannot occur if we restrict our attention to the special case where all the vertex semigroups are isomorphic to $\N$, that is, right-angled Artin monoids $S(\Lambda)=: A_\Lambda^+$. Here, the essential obstruction to the existence of the generalized scale in the coconnected case is the presence of edges, as $\Gamma(A_\Lambda^+) \cong \Lambda$, see Theorem~\ref{thm:char existence of the GS}. This already follows from \cite{Sta4}*{Corollary~4.9}, though the approach taken there does not provide for a workaround for graphs with edges. In this respect, Theorem~\ref{thm:char existence of the GS} is a clear improvement: If $\Lambda$ is a coconnected graph with at least one edge (hence at least two vertices), then $A_\Lambda^+$ satisfies all but condition (ii) of Theorem~\ref{thm:char existence of the GS}. 

\subsection{Saito's degree map}\label{subsec:saito}
If $S$ admits a generalized scale $N\colon S \to \N^\times$, applying the logarithm results in a homomorphism $\text{deg}\colon S \to \R_{\geq 0}$. Under the assumption that the core $S_c$ is nothing more than the invertible elements $S^*$ of $S$, the map $\text{deg}$ happens to be a particular case of a \emph{degree map} in the sense of \cite{Sai}*{Section~4}. Degree maps are employed in \cite{Sai} to define a growth function and the skew-growth function for the monoid, which are then shown to be inverse to each other as Dirichlet series. In the notation of \cite{Sai} with a change of variable $t\mapsto e^{-\beta}$, we get 
\[\begin{array}{c} P_{\text{deg}}(\beta) = \sum\limits_{[s] \in S/_\sim} N_s^{-\beta} = \sum\limits_{n \in N(S)} n^{1-\beta} = \zeta_S(\beta), \end{array}\]
the $\zeta$-function appearing in \cite{ABLS}*{Definition~4.2}. The skew-growth function appears in \cite{ABLS}*{Remark~7.4} and is crucial for the reconstruction formula for KMS$_\beta$-states on $C^*(S)$ in \cite{ABLS}*{Lemma~7.5}. With this perspective and the inversion formula established in \cite{Sai}, degree maps (with extra properties related to the intersection of right ideals) may yield interesting dynamics on $C^*$-algebras associated to cancellative monoids, for which the structure of KMS-states can be studied to a satisfactory degree. The assumption $S_c=S^*$ from above is needed to harmonize the two approaches as Saito identifies elements in the monoid up to $S^*$, while an identification up to $S_c$ is more natural in the approach pursued in \cite{ABLS}. It may thus be very interesting to see how much of Saito's theory can be transferred from $S^*$ to $S_c$.



\end{document}